\documentclass[reqno]{amsart}

\calclayout

\usepackage[unicode]{hyperref}
\usepackage[lite]{amsrefs}
\usepackage{amsmath,amssymb,latexsym,mathtools}
\let\zeroslash\emptyset
\usepackage{graphicx}
\usepackage{subcaption}		
\usepackage{enumitem}

\def\paragraph{\subsection}
\makeatletter
\def\@sect#1#2#3#4#5#6[#7]#8{%
  \edef\@toclevel{\ifnum#2=\@m 0\else\number#2\fi}%
  \ifnum #2>\c@secnumdepth \let\@secnumber\@empty
  \else \@xp\let\@xp\@secnumber\csname the#1\endcsname\fi
  \@tempskipa #5\relax
  \ifnum #2>\c@secnumdepth
    \let\@svsec\@empty
  \else
    \refstepcounter{#1}%
    \edef\@secnumpunct{%
      \ifdim\@tempskipa>\z@ 
        \@ifnotempty{#8}{.\@nx\enspace}%
      \else
        \@ifempty{#8}{.}{.\@nx\enspace}%
      \fi
    }%
      \ifnum #2=\tw@ \def\@secnumfont{\bfseries}\fi{}%
    \protected@edef\@svsec{%
      \ifnum#2<\@m
        \@ifundefined{#1name}{}{%
          \ignorespaces\csname #1name\endcsname\space
        }%
      \fi
      \@seccntformat{#1}%
    }%
  \fi
  \ifdim \@tempskipa>\z@ 
    \begingroup #6\relax
    \@hangfrom{\hskip #3\relax\@svsec}{\interlinepenalty\@M #8\par}%
    \endgroup
    \ifnum#2>\@m \else \@tocwrite{#1}{#8}\fi
  \else
  \def\@svsechd{#6\hskip #3\@svsec
    \@ifnotempty{#8}{\ignorespaces#8\unskip
       \@addpunct.}%
    \ifnum#2>\@m \else \@tocwrite{#1}{#8}\fi
  }%
  \fi
  \global\@nobreaktrue
  \@xsect{#5}}
\makeatother

\makeatletter
\def\pxspace{\@ifnextchar.{\@}{.\@\xspace}}
\makeatother

\newtheorem{theorem}{Theorem}[section]
\newtheorem{assumption}[theorem]{Assumption}
\newtheorem{case}[theorem]{Case}

\newtheorem{definition}[theorem]{Definition}
\newtheorem{lemma}[theorem]{Lemma}

\newtheorem{remark}[theorem]{Remark}

\newcounter{figuregroup}
\newenvironment{figuregroup}{%
  \global\let\thefiguresave=\thefigure
  \global\let\theHfiguresave=\theHfigure
  \refstepcounter{figure}
  \numberwithin{figure}{figuregroup}
  \setcounter{figuregroup}{\value{figure}}
  \setcounter{figure}{0}
}{%
  \global\let\thefigure=\thefiguresave
  \global\let\theHfigure=\theHfiguresave
  \setcounter{figure}{\value{figuregroup}}
}

\makeatletter
\newcommand{\manuallabel}[2]{\edef\@currentlabel{#2}\label{#1}}
\makeatother
\newcommand{\itemlabel}[2]{\item[#1]\manuallabel{#2}{#1}}

\newcommand{\interject}[1]{\noalign{\begin{quote}#1\end{quote}}}

\newcommand{\disconnect}{\leavevmode\par}

\newcommand{\forwardref}[1]{\ref{#1}}

\newcommand{\mbbCP}{\mathbb{CP}}
\newcommand{\mbbN}{\mathbb{N}}

\newcommand{\mbbR}{\mathbb{R}}
\newcommand{\mbbSS}{\mathbb{SS}}
\newcommand{\mbbSY}{\mathbb{SY}}
\newcommand{\mbbY}{\mathbb{Y}}
\newcommand{\mbbYF}{\mathbb{YF}}
\newcommand{\mbbZ}{\mathbb{Z}}

\newcommand{\scrL}{\mathcal{L}}

\newcommand{\zerohat}{\hat{0}}

\newcommand{\Idealf}{{\operatorname{Ideal}_{\textup{f}}}}

\newcommand{\tnew}{\textup{new}}
\newcommand{\tmax}{\textup{max}}
\newcommand{\tmin}{\textup{min}}

\newcommand{\upcolon}{\textup{:}}

\newcommand{\mvert}{\,|\,}      

\begin{document}

\newcommand{\ipeFigJXB}{1}
\newcommand{\ipeFigNzerozerozero}{2}
\newcommand{\ipeFigNzeroonezero}{3}
\newcommand{\ipeFigNzerotwozero}{4}
\newcommand{\ipeFigNonezerozero}{5}
\newcommand{\ipeFigNoneonezero}{6}
\newcommand{\ipeFigNoneoneone}{7}
\newcommand{\ipeFigNonetwozero}{8}
\newcommand{\ipeFigNonetwoone}{9}
\newcommand{\ipeFigNtwozerozero}{10}
\newcommand{\ipeFigNtwoonezero}{11}
\newcommand{\ipeFigNtwooneone}{12}
\newcommand{\ipeFigNtwotwozero}{13}
\newcommand{\ipeFigNtwotwoone}{14}
\newcommand{\ipeFigNtwotwotwo}{15}
\newcommand{\ipeFigInfinitePath}{16}
\newcommand{\ipeFigFinitePath}{17}
\newcommand{\ipeFigTriple}{18}
\newcommand{\ipeFigFucJ}{19}
\newcommand{\ipeFigFucL}{20}
\newcommand{\ipeFigU}{21}
\newcommand{\ipeFigAssignment}{22}
\newcommand{\ipeFigSimEquiv}{23}
\newcommand{\ipeFigCthreetwo}{24}
\newcommand{\ipeFigCfourtwo}{25}
\newcommand{\ipeFigCsixtwo}{26}
\newcommand{\ipeFigTwosPoints}{27}
\newcommand{\ipeFigTwosLattice}{28}
\newcommand{\ipeFigYP}{29}
\newcommand{\ipeFigYL}{30}
\newcommand{\ipeFigSYP}{31}
\newcommand{\ipeFigSYL}{32}
\newcommand{\ipeFigYthreeP}{33}
\newcommand{\ipeFigYthreeL}{34}

\title[Classification of lattices]{On the combinatorics of tableaux ---
  Classification of lattices underlying Schensted correspondences}
\author{Dale R. Worley}
\email{worley@alum.mit.edu}
\date{Jan 11, 2026} 

\begin{abstract}
The celebrated Robinson--Schensted algorithm and each of its variants
that have attracted substantial attention can be constructed using
Fomin's ``growth diagram'' construction from a modular lattice that is
also a weighted-differential poset.  In this paper we classify all
such ``Fomin'' lattices that meet certain criteria; the main criterion
is that the lattice is distributive.  Intuitively, these criteria seem
excessively strict, but all known Fomin lattices satisfy all of these
criteria, with the sole exception of one family that is not even
distributive, the Young--Fibonacci lattices and cartesian products
involving them.  We discover a new class of Fomin lattices, but
unfortunately they cannot be used to construct Robinson--Schensted
algorithms.
\end{abstract}



\maketitle

\textit{Note added to the 2\textsuperscript{nd} version:}
Paragraphs 6.5, 6.6, and 6.7 of the 1\textsuperscript{st} version were
incorrect and have been replaced.  Their conclusions remain unchanged.
Also, a new class of Fomin lattices has been added as item (11) of
para. 2.3:  the ``vertical'' joining of an arbitrary doubly-inﬁnite
sequence boolean algebras $B_1$ or $B_2$.

\tableofcontents

\section{Introduction} \label{sec:intro}

The celebrated Robinson--Schensted
algorithm\cite{Rob1938a}\linebreak[0]\cite{Schen1961a}\linebreak[0]%
\cite{Fom1994a}\linebreak[0]\cite{Fom1995a}
and each of its variants that have attracted substantial attention
can be constructed\cite{Wor2023c}*{sec.~2} using Fomin's ``growth
diagram'' construction from a modular lattice that is also a
(weighted-)differential poset\cite{Stan1988a}.  We define these
requirements as a ``Fomin lattice'' that has a minimum element
$\zerohat$.
In this paper we classify all Fomin lattices that meet certain criteria.
Our main criterion is that the lattice is distributive,
with secondary criteria on the set of points (the join-irreducible
elements of the lattice).  Intuitively, these criteria seem excessively
strict, but all known Fomin lattices satisfy all of these criteria,
with the sole exception of one family that is not even distributive,
the Young--Fibonacci lattices.  This work adds to the classes of
lattices within which all of the Fomin lattices have been classified.

We find a class of Fomin lattices that appear to be new to the
literature:  Let $L$ be the ``vertical'' joining of an arbitrary
doubly-infinite sequence of 2-chains or $B_1$ (boolean algebras on 1
element), and diamonds or $B_2$ (boolean algebras on 2 elements).  Then
$L$ with the weighting is Fomin, but unfortunately has no $\zerohat$.
See an example in fig.~\ref{fig:twos} and
para.~\ref{para:known}(\forwardref{i:previous:twos}).

The history of this problem in the literature is discussed in
para.~\ref{para:literature}.
Para.~\ref{para:properties} enumerates the relevant properties of structures.
All known Fomin lattices are described in para.~\ref{para:known}.
They consist of Young's lattice of partitions, the shifted Young's
lattice of strict partitions, the lattices of partitions into a limited
number of parts, the Young-Fibonacci lattices, the cylindric partition
lattices, the lattice of ideals of a ``skew strip'',
the new class of lattices, some trivial
cases, and certain derivatives of other Fomin lattices.
The previous classification theorems for Fomin lattices are discussed
in para.~\ref{para:prev-classification}.

Our new classification theorem is in sec.~\ref{sec:classification}.
The class of lattices classified by
our theorem is incomparable with the classes of previous
classification theorems; it does not supersede them.
However, it is an incremental contribution to the classification of
Fomin lattices (and thus Robinson--Schensted algorithms) and is the
only one that applies to ``weighted-differential'' lattices, that is,
where the weight on coverings (or points) is not identically 1.

Our proof techniques are elementary.
They consist of a careful accounting of the consequences of the
differential condition~(\forwardref{eq:differential}) in lattices that
satisfy our criteria.  This ``local characterization''
(sec.~\ref{sec:local}) includes a succession of restatements equivalent to
the differential condition (lem.~\forwardref{lem:equiv-1},
\forwardref{lem:equiv-2}, \forwardref{lem:equiv-3},
\forwardref{lem:equiv-4}, \forwardref{lem:equiv-4-dual},
\forwardref{lem:equiv-5},
\forwardref{lem:equiv-6}, and \forwardref{lem:equiv-6-dual}).
These results culminate in lem.~\forwardref{lem:no-triples}
and~\forwardref{lem:no-triples-dual},
that no point covers three distinct points or is covered by three
distinct points.
We then proceed to the ``global characterization'' (sec.~\ref{sec:global})
of how the points can be placed on a two-dimensional grid.  This
culminates in the
classification theorem, th.~\forwardref{th:classification}.

Finally, we discuss possible future directions for research
(sec.~\ref{sec:future}), which
assesses the prospect of various relaxations of the criteria of the
classification theorem.

Thus this article consists of:
\begin{itemize}
\item background information including the previous literature, a
  catalog of properties, known Fomin lattices, and a recapitulation of
  properties of the Cartesian product (sec.~\ref{sec:background}),
\item a statement of the definitions essential for the classification
  (sec.~\ref{sec:fomin}),
\item a construction supporting the discussion of future directions
  (sec.~\ref{sec:aux}),
\item the characterization of Fomin lattices that meet our criteria
  (sec.~\ref{sec:local} and~\ref{sec:global}),
\item the classification theorem for Fomin lattices that meet our
  criteria (sec.~\ref{sec:classification}), and
\item a discussion of future directions for research
  (sec.~\ref{sec:future}).
\end{itemize}

The development essential to prove our classification
theorem~\forwardref{th:classification} is contained in
sec.~\ref{sec:fomin}, \ref{sec:local}, \ref{sec:global},
and~\ref{sec:classification}.

\section{Background} \label{sec:background}

\paragraph{Previous literature} \label{para:literature}
The literature on differential structures can be classified based on
the generality
of the structures being considered.  The generality can be described
along two dimensions:
\begin{enumerate}
\item the
generality of the underlying structures (lattice, poset, etc.) being
considered and
\item whether the covering pairs of elements of the
structure are all given the weight 1 (the \emph{unweighted} case) or
are allowed varying weights (the \emph{weighted} case).
\end{enumerate}
In regard to the underlying structures, from most general to least
general, the major categories are
\begin{enumerate}
\item dual graded graph (two graphs or multi-graphs on the same set
  of vertices with the same grading with separate weights or edge
  multiplicities in the two graphs between pairs of points),
\item graded poset,
\item lattice (which by the differential property must be modular),
\item distributive lattice, and
\item distributive lattice with the further conditions of
  th.~\forwardref{th:classification}.
\end{enumerate}

Significant works in the literature, and distinctive parts of them,
are classified in the following table.
In some cases a work's definitions and general analysis are
considerably more general than
all of the examples described in the work; in that case the
classification of the definitions are marked ``D'' and classification
of the examples are marked ``E''.
Some works that are largely a description of examples (often novel
examples) are only marked with ``E''.
The works with the best previous classification results (detailed at
the end of para.~\ref{para:known}) are listed with ``C'' in the
appropriate category.

\vspace{0.5\baselineskip}
{
\renewcommand{\arraystretch}{1.5}
\newcommand{\R}{\raggedright}
\newcommand{\cwd}{0.29\textwidth}
\begin{tabular}{|l|p{\cwd}|p{\cwd}|p{\cwd}|}
\cline{3-4}
\multicolumn{2}{c|}{} & \multicolumn{2}{c|}{Weighting} \\
\cline{3-4}
\multicolumn{2}{c|}{} & Weighted & Unweighted \\
\hline
\raisebox{-0.5em}[\height][0pt]{%
\rotatebox[origin=Br]{90}{Underlying structure}%
} &
Dual graded graph \R & Fomin \cite{Fom1994a}*{Def.~1.3.4} \R &
Qing \cite{Qing2008a}*{sec.~4.2} E \\
\cline{2-4}
& Graded poset \R & & Stanley \cite{Stan1988a}*{Def.~1.1} DC \newline
Lewis \cite{Lew2007a}*{sec.~3} E \\
\cline{2-4}
& Lattice (modular) \R & Stanley \cite{Stan1990a}*{sec.~3} \newline
this article D \R & Stanley \cite{Stan1988a} E \newline
Byrnes \cite{Byrn2012a}*{Ch.~6} C \\
\cline{2-4}
& Distributive lattice \R & Fomin \cite{Fom1994a}*{sec.~2.2} &
Elizalde \cite{El2025a} E \\
\cline{2-4}
& Distributive lattice with further restrictions \R & this article E \R & \\
\hline
\end{tabular}
}

\paragraph{Properties considered in this article} \label{para:properties}
A large number of properties of structures are considered in this
area of combinatorics.  This section catalogs the
properties that are relevant to this article and their definitions.
This article classifies in sec.~\ref{sec:classification} structures
which have all of these properties.
The properties are here generally ordered from weakest to strongest.
Some theorems relating the properties are included.
For a casual reading of this article, this section can be skipped and
referred back to as necessary.
Section~\ref{sec:fomin} gives terse definitions
that suffice for reading characterizations in this article.

\begin{enumerate}
\itemlabel{(A)}{i:prop-DGG} \emph{The structure is a dual graded graph (DGG):}
\cite{Fom1994a}*{Def.~1.3.4} There is a graded set of vertices.  In the original
formulation each grade
contains a finite number of vertices, but we do not include that
criterion here.  (But see property~\forwardref{i:prop-fin-covers}.)
Vertices in a grade
have edges to some vertices in adjacent grades, forming a graph.
Each edge has an assigned weight, with the weights being drawn from a
set of \emph{values} $V$.
In the original formulation, each edge is
an otherwise unlabeled multi-edge, so the weights
are positive integers ($V = \mbbZ$), but we do not include that criterion here.

A dual graded graph is two such graphs on the same graded set of
vertices, and thus there are up to two weights between any pair
of vertices in adjacent grades.  The edges of the first graph are
conceptualized as being directed upward (toward higher grades) and the
edges of the second graph are conceptualized as being directed
downward.

There is some ambiguity in the presentation of a DGG; if a pair
of vertices has a non-zero
weight between them in only one graph, we may add an edge
between them in the other graph that has a weight of 0 without
changing any of the essential properties of the DGG.  This allows
the two graphs to have the same set of edges, with the two graphs
differing only in the weights of the edges.

Critically, the DGG has the \emph{differential property}:  It has a
constant $r$ called the \emph{differential degree} and the DGG is
called \emph{$r$-differential}.
Assume that weight 0 edges have been added so that both graphs have
the same edges. Let $x \lessdot y$
mean that $y$ is in the grade above $x$ and $(x,y)$ is an edge.
When $x \lessdot y$, let
$w_1(x \lessdot y)$ be the weight of the edge $(x,y)$ in the first
graph and $w_2(x \lessdot y)$ be the weight of the edge $(x,y)$ in the
second graph.
Then the differential property is that for every pair of vertices $x$
and $y$ in the same grade,

\begin{equation} \label{eq:DGG-differential}
\sum_{z \mvert x \lessdot z, y\lessdot z}  w_1(x \lessdot z) w_2(x \lessdot z)
- \sum_{v \mvert v \lessdot x, v \lessdot y}  w_2(v \lessdot x) w_1(v \lessdot y)
=
\begin{cases}
0 & \text{if } x \neq y \\
r & \text{if } x = y
\end{cases}
\end{equation}

\itemlabel{(B)}{i:prop-diff} \emph{The structure is a differential poset:}
\cite{Stan1988a} This is a narrower type of DGG.  Both graphs have the
same set of edges and the weights on the edges of the second graph are
all 1.  In the original formulation, the weights on the edges of the
first graph are also 1, but we do not include that criterion here.
(But see property~\forwardref{i:prop-0}.)
Differential posets with weights that are not 1 are also called
\emph{weighted-differential} posets.

The graphs are the Hasse diagram of a poset.
In the original formulation, the poset is assumed to have a minimum
element, $\zerohat$, but we do not include that criterion here.
A differential poset is required to be \emph{unique-cover-modular}.

\end{enumerate}

\begin{definition} \label{def:cover-modular}
A poset $P$ is \emph{cover-modular}\footnote{%
We choose this name because of its resemblance to the property of
coverings in modular lattices.  But in general, $P$ is not necessarily
modular, indeed, it need not even be a lattice.
A poset with this property and finite length is called a ``modular
poset'' in \cite{Birk1967a}*{\S II.8}.}
if when $x, y \in P$ and $x$
and $y$ both cover an element of $P$, then they are both covered by an
element of $P$ (which may not be unique).  Dually, we require
that when $x, y \in P$ and $x$
and $y$ both are covered by an element of $P$, then they both cover an
element of $P$ (which may not be unique).
\end{definition}

\begin{definition} \label{def:unique-cover-modular}
A poset $P$ is \emph{unique-cover-modular}%
\footnote{This property is called a ``uniquely modular poset'' in
\cite{Proct1982a}*{sec.~5}.}
if when $x, y \in P$ and
either $x$ and $y$ both cover an element of $P$ or
$x$ and $y$ are covered by an element of $P$, then
they both cover a unique element of $P$
and they are both covered by a unique element of $P$.
\end{definition}

\begin{theorem} \cite{Birk1967a}*{\S II.8 Th.~14} \label{th:graded}
  If a locally-finite poset has $\zerohat$ and is cover-modular, then
  it is graded.
\end{theorem}

\begin{enumerate}[resume]

\itemlabel{(C)}{i:prop-loc-fin} \emph{The structure is locally-finite:}
In both the DGG and differential poset formulations, the
structure (seen as a poset) is assumed to be locally-finite.

\itemlabel{(D)}{i:prop-fin-covers} \emph{The structure has finite covers:}
In both the DGG and differential poset formulations, every
vertex has a finite number of upward and downward edges, or a finite
number of upward and downward covers.

\cite{Fom1994a}*{sec.~1.1} requires each grade of vertices to be
finite.  If the structure has
a $\zerohat$, this is equivalent to finite covers, but without a
$\zerohat$, it is a stronger property.  (See also
property~\forwardref{i:prop-0}.)

\itemlabel{(E)}{i:prop-umc} \emph{The structure has unique-modular-covers and the
weights are projective-constant:}
In this case, if $x \neq y$,
(\ref{eq:DGG-differential}) is satisfied automatically.
If we define $w(x \lessdot y) = w_1(x \lessdot y) w_2(x \lessdot y)$,
then the $x = y$ case of (\ref{eq:DGG-differential}) is equivalent to:
For every vertex $x$,

\begin{equation} \label{eq:DP-differential}
\sum_{z \mvert x \lessdot z} w(x \lessdot z)
- \sum_{v \mvert v \lessdot x}  w(v \lessdot x)
= r
\end{equation}

Thus if we require unique-modular-covering and projective-constancy,
the DGG and weighted-differential poset structures effectively
coincide with only one weight function.
Differential posets that are not weighted-differential are automatically
projective-constant.
\end{enumerate}

\begin{definition} \label{def:projective-constant}
The weights on a DGG or differential poset are
\emph{projective-constant} iff (for a DGG) the weight functions $w_1$
and $w_2$ or (for a differential poset) the weight function $w$ have
the following property:
the function $f$ is equal on
coverings which are \emph{cover-projective}:
If $x \vee y$ covers $x$ and $y$ and $x \wedge y$ is covered by $x$ and $y$,
then $f(x \wedge y \lessdot x) = f(y \lessdot x \vee y)$
and $f(x \wedge y \lessdot y) = f(x \lessdot x \vee y)$.
\end{definition}

\begin{enumerate}[resume]
\itemlabel{(F)}{i:prop-lattice} \emph{The structure is a differential poset
that is a lattice.}
\end{enumerate}

\begin{lemma}
\cite{Birk1967a}*{\S II.8 Th.~16}
If a lattice is locally-finite and cover-modular, then it is modular.
\end{lemma}

\begin{enumerate}[resume]
\itemlabel{(G)}{i:prop-mod-lattice} \emph{The structure is a differential
poset that is a modular lattice:}
If the structure is a locally-finite modular lattice, then
if the weight function is projective-constant, it is
constant on projectivity classes\cite{Birk1967a}*{\S I.7 Def.} of coverings.
\end{enumerate}

We define ``Fomin lattices'' to be structures with the above set of
properties.

\begin{enumerate}[resume]
\itemlabel{(H)}{i:prop-0} \emph{The structure has a $\zerohat$:}
We consider the existence of a minimum vertex as a separate property.
\cite{Stan1988a} defines all differential posets to have a $\zerohat$.

\itemlabel{(I)}{i:prop-positive} \emph{The weighting and differential
degree are positive.}
\end{enumerate}

\begin{definition} \label{def:positive}
A weighting $w$ and differential degree $r$ (taken together)
are \emph{positive}
if the set of values $V$ (the target set of $w$) is $\mbbZ$, all of
the values of $w$ are positive, and $r$ is positive.
\end{definition}

Depending on the degree of abstraction of the discussion, the
requirements for the set of values $V$ varies.
The minimum requirement is that $V$ be a module over a unitial ring of
\emph{scalars} $S$.
Usually we will be unconcerned with the details of
$S$ other than it necessarily contains a homomorphic image of $\mbbZ$
and thus we can multiply elements of $V$ by integer coefficients.
If $V = \mbbZ$, then $\mbbZ$ is \emph{isomorphically}
embedded in $S$.

\begin{lemma} \label{lem:w-positive}
If $L$ is a Fomin lattice with $\zerohat$, $L$ is not the trivial
(one-element) lattice, $V = \mbbZ$, and all values of $w$ are
positive, then $r$ is positive.
\end{lemma}
\begin{proof} \disconnect
For $x = \zerohat$, (\ref{eq:DP-differential}) becomes
$r = \sum_{z \mvert x\lessdot z} w(\zerohat \lessdot z)$.
But by hypothesis, there is at least one $z$ in the range of the sum and
$w(\zerohat \lessdot z) > 0$, which ensures
the sum is positive.  That proves $r$ is positive.
\end{proof}

In light of the above lemma,
we will abuse language by using ``$w$ is positive'' or ``the lattice
is positive'' to mean that the
weighting and differential degree are positive, that is, $V$ is
$\mbbZ$, $r > 0$, and all of
the values of $w$ are positive.

\begin{enumerate}[resume]
\itemlabel{(J)}{i:prop-dist} \emph{The structure is a distributive lattice:}
In this case, we call the lattice $L$ and define $P$ to be the poset
of \emph{points} (join-irreducibles) of $L$.
If the lattice is finitary, it is
isomorphic to the set of finite order ideals of $P$.

\itemlabel{(K)}{i:prop-points-ucm} \emph{The structure is a lattice
and the  poset of points $P$ is unique-cover-modular:}
Note that this requirement is on the poset of points, not
a requirement on the structure itself (which all Fomin lattices satisfy).

\itemlabel{(L)}{i:prop-factor} \emph{The structure is a lattice that
cannot be factored into two nontrivial lattices:}
If the lattice is distributive this is equivalent
to that $P$ cannot be partitioned into two disjoint non-empty
posets $P = P_1 \sqcup P_2$.%
\footnote{We use $\sqcup$ for disjoint unions of sets and posets
because ``disjoint union'' is the coproduct in the categories of sets
and posets.}
\end{enumerate}

\paragraph{Known Fomin lattices} \label{para:known}
All known Fomin lattices can be constructed as follows:
\begin{enumerate}
\item \label{i:previous:degenerate}
  Any lattice satisfying property~\ref{i:prop-DGG} can be made into a
  \emph{degenerate} Fomin lattice by setting $r = 0$ and $w = 0$.
\item \cite{Fom1994a}*{Exam.~2.1.2 and~2.2.7}
  \label{i:previous:Y}
  Young's lattice $\mbbY$, the lattice of
  partitions, with $r=1$ and $w=1$ is a positive Fomin lattice with
  $\zerohat$.  See fig.~\ref{fig:Y}.
  (There are additional weightings $w$ on this lattice that make it Fomin but
  are not positive.)

\begin{figure}[htp]
\centering
\begin{subfigure}[b]{0.45\textwidth}
  \centering
  \includegraphics[page=\ipeFigYL,scale=0.5]{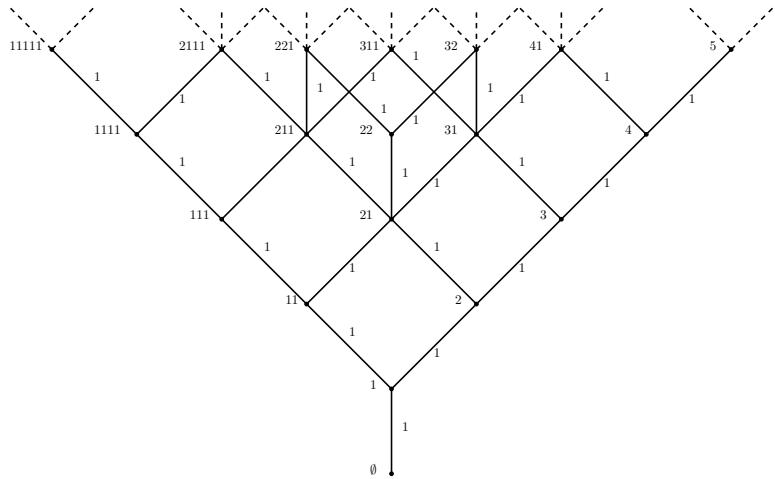}
  \caption{Lattice $\mbbY$ with weights of the coverings}
  \label{fig:YL}
\end{subfigure}
\\
\begin{subfigure}[b]{0.45\textwidth}
  \centering
  \includegraphics[page=\ipeFigYP,scale=1]{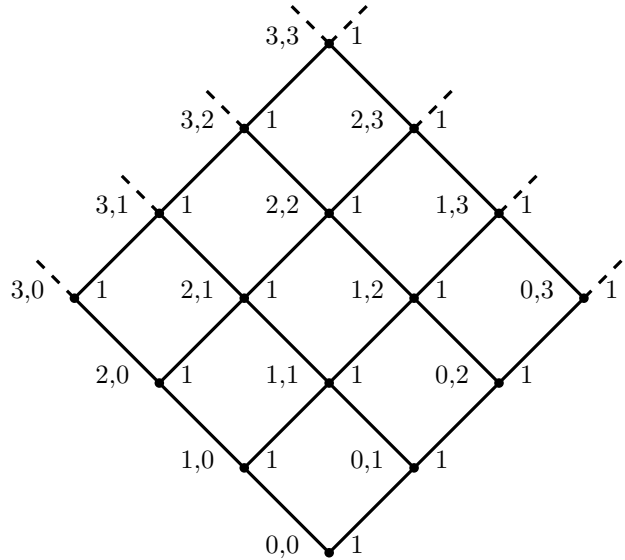}
  \caption{The quadrant $J(\mbbY)$ with weights of the points}
  \label{fig:YP}
\end{subfigure}
\caption{Young's lattice $\mbbY$, the lattice of partitions, with its
  canonical weighting.}
\label{fig:Y}
\end{figure}

\item \cite{Sag1979b}\cite{Wor1984a}\cite{Fom1994a}*{Exam.~2.2.8}
  \label{i:previous:SY}
  The shifted Young's Lattice $\mbbSY$, the lattice of
  partitions into unequal parts, with $r=1$ and
  $w = 1 \textup{ or } 2$ is a positive Fomin lattice with
  $\zerohat$.  See fig.~\ref{fig:SY}.

\begin{figure}[htp]
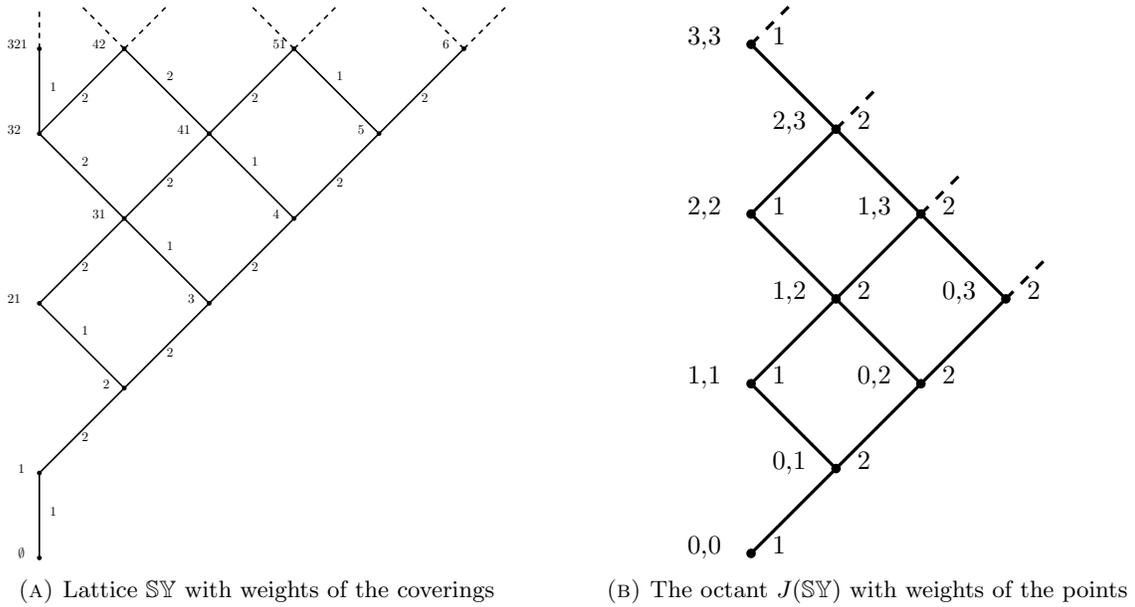

\centering
\begin{subfigure}[b]{0.45\textwidth}
  \centering
  \includegraphics[page=\ipeFigSYL,scale=0.5]{k-row-figs.pdf}
  \caption{Lattice $\mbbSY$ with weights of the coverings}
  \label{fig:SYL}
\end{subfigure}
\begin{subfigure}[b]{0.45\textwidth}
  \centering
  \includegraphics[page=\ipeFigSYP,scale=1]{k-row-figs.pdf}
  \caption{The octant $J(\mbbSY)$ with weights of the points}
  \label{fig:SYP}
\end{subfigure}
\caption{The shifted Young's lattice $\mbbSY$, the lattice of strict
  partitions, with its canonical weighting.}
\label{fig:SY}
\end{figure}

\item \cite{Fom1994a}*{Exam.~2.2.5} \label{i:previous:Yk}
  For any integer $k \geq 1$, the $k$-row Young's lattice,
  $\mbbY_k$, the lattice of partitions into at most $k$ parts (or
  dually, into an arbitrary number of parts that are all
  $\leq k$), with $r = k$ and $w$ having a particular range of values
  is a positive Fomin lattice with
  $\zerohat$.  See $\mbbY_3$ in fig.~\ref{fig:Y3}.

\begin{figure}[htp]
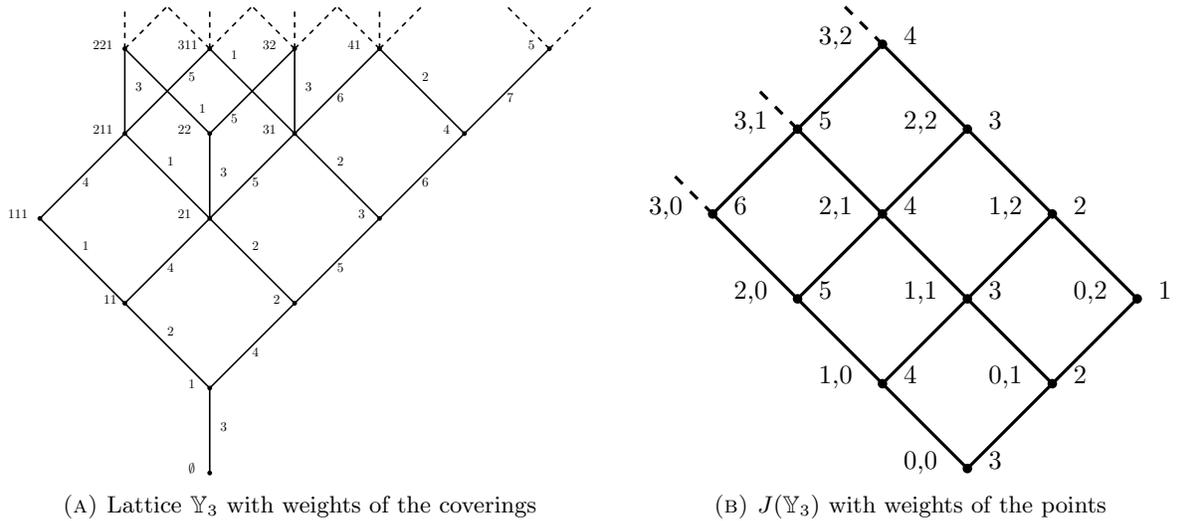

\centering
\begin{subfigure}[b]{0.45\textwidth}
  \centering
  \includegraphics[page=\ipeFigYthreeL,scale=0.5]{k-row-figs.pdf}
  \caption{Lattice $\mbbY_3$ with weights of the coverings}
  \label{fig:Y3L}
\end{subfigure}
\begin{subfigure}[b]{0.45\textwidth}
  \centering
  \includegraphics[page=\ipeFigYthreeP,scale=1]{k-row-figs.pdf}
  \caption{$J(\mbbY_3)$ with weights of the points}
  \label{fig:Y3P}
\end{subfigure}
\caption{The lattice $\mbbY_3$ of partitions with $\leq 3$ rows, with
  its canonical weighting.}
\label{fig:Y3}
\end{figure}

\item \cite{Fom1994a}*{Exam.~2.2.1} \label{i:previous:N}
  The upward semi-infinite chain $\mbbN$
  with $r = k$ and $w(i) = k(i+1)$ (which is isomorphic to $\mbbY_1$, a special
  case of (\ref{i:previous:Yk})) is a positive Fomin lattice with
  $\zerohat$.
\item \cite{Fom1994a}*{Exam.~2.2.9} \label{i:previous:Z}
  The infinite chain $\mbbZ$ with and $w(i)=ri$ is a Fomin lattice
  but has no $\zerohat$ and is not positive.
\item \cite{Stan1988a}*{sec.~5}\cite{Roby1991a}*{Ch.~5}
  \label{i:previous:YF}
  For any integer $k \geq 1$, the Young--Fibonacci
  lattice of degree $k$, $\mbbYF_k$,
  with $r = k$ and $w=1$ is a positive Fomin lattice with
  $\zerohat$.
\item \cite{El2025a}*{sec.~5} \label{i:previous:CP}
  For any integers $d, L \geq 1$,
  the \emph{cylindric partition} lattice $\mbbCP_{dL}$ with $r = 0$
  and $w = 1$ is a Fomin lattice but is not
  positive and has no $\zerohat$.\footnote{%
  \cite{El2025a} uses the notation $\scrL_{d,L}$ for $\mbbCP_{dL}$.
  The cylindric partitions have been known since at least
  \cite{GessKratt1997a} but \cite{El2025a} seems to be the first
  report that all of the $\mbbCP_{dL}$ are 0-differential.
  See (\forwardref{i:previous:skewstrip}):
  \cite{Fom1994a} reports that the $\mbbCP_{2t}$ are
  0-differential.}
\item \cite{Fom1994a}*{Exam.~2.2.10} \label{i:previous:skewstrip}
  defines $t\textit{-SkewStrip}$, which is a special case of
  (\ref{i:previous:CP}):  $t\textit{-SkewStrip} = \mbbCP_{2t}$.
\item \cite{Fom1994a}*{Exam.~2.2.11} \label{i:previous:SS}
  $\mbbSS_t$ is the lattice of non-trivial ideals of
  $t\textit{-SkewStrip} = \mbbCP_{2t}$, those that are neither empty
  nor all of $t\textit{-SkewStrip}$,
  with $r=0$ and  $w = 1 \textup{ or } 2$ is a
  Fomin lattice but is not positive and has no $\zerohat$.%
\item \label{i:previous:twos}
  Let $L$ be the ``vertical'' joining of an arbitrary doubly-infinite sequence of
  2-chains or $B_1$ (boolean algebras on 1 element), and diamonds or
  $B_2$ (boolean algebras on 2 elements).  Then $L$ with
  $r=0$ and $w = 1 \textup{ or } 2$ is a Fomin lattice but is not
  positive and has no $\zerohat$.  See an example in fig.~\ref{fig:twos}.
\begin{figure}[htp]
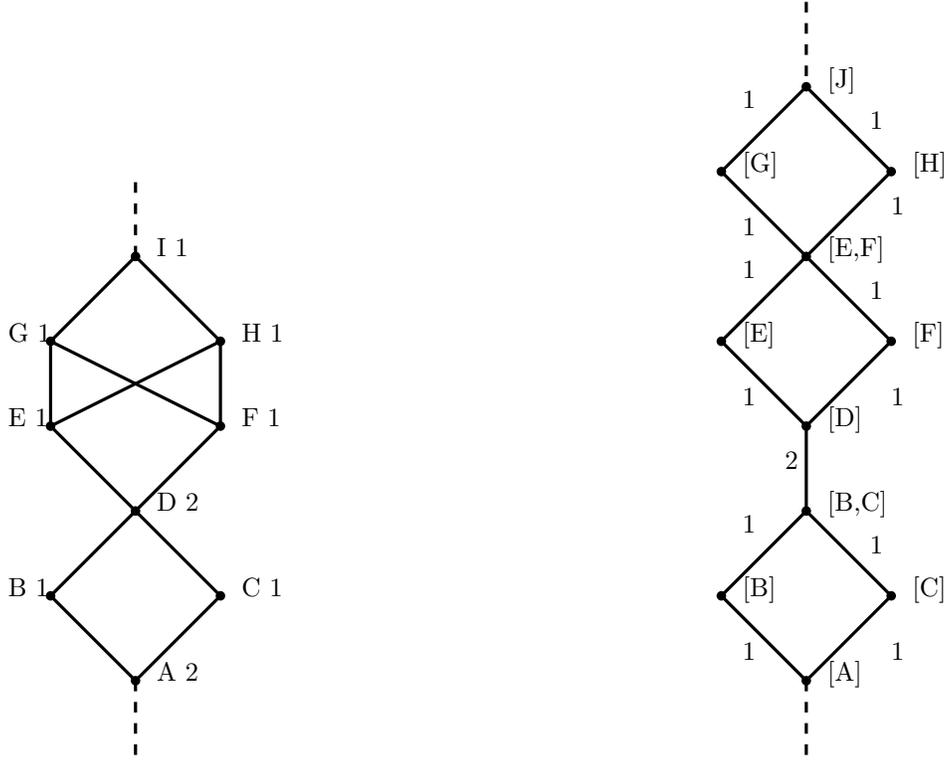

\centering
\begin{subfigure}[b]{0.45\textwidth}
  \centering
  \includegraphics[page=\ipeFigTwosPoints,scale=1]{k-row-figs.pdf}
  \caption{A section of the poset of points, with the weights of the points}
\end{subfigure}
\hspace{0.05\textwidth}
\begin{subfigure}[b]{0.45\textwidth}
  \centering
  \includegraphics[page=\ipeFigTwosLattice,scale=1]{k-row-figs.pdf}
  \caption{A section of the lattice, with the weights of the
    coverings, labeled with the generators of the ideals}
\end{subfigure}
\caption{Example of a Fomin lattice of type (\ref{i:previous:twos})}
\label{fig:twos}
\end{figure}
\item \label{i:previous:cartesian} \cite{Fom1994a}*{Lem.~2.2.3}
  Any cartesian product of a finite number of Fomin lattices
  that all have the same value set $V$,
  with the weight of a covering in the product lattice equal to the weight
within the factor from which the covering is derived and the
differential degree equal to the sum of the differential degrees of the
factors, is a Fomin lattice.
The product lattice has $\zerohat$ if all factors have $\zerohat$ and
is positive if all factors are positive.
\item \label{i:previous:multiply} Given a Fomin lattice,
  the same lattice but multiplying the differential degree and
  weighting by the same constant gives a Fomin lattice.
  The resulting lattice is positive if the original lattice is
  positive and  the constant is a positive integer.
\end{enumerate}

\paragraph{Previous classifications} \label{para:prev-classification}
The previous classifications of Fomin lattices with $\zerohat$ are:
\begin{enumerate}
\item \cite{Stan1988a}*{Prop.~5.5}
  If a Fomin lattice with $\zerohat$ is distributive
  and has $w=1$, then it is isomorphic to $\mbbY^{\times r}$.%
\footnote{We use $\bullet^{\times n}$ to denote the $n$-fold cartesian
exponential to avoid ambiguity with other uses of superscripts.}
\item \cite{Byrn2012a}*{Th.~1.3}
If a Fomin lattice with $\zerohat$ has $r=1$ and $w=1$, then it is
either $\mbbY$ or $\mbbYF_1$.
\end{enumerate}
In particular, there are no previous classifications for the weighted
case, when $w$ is not identically $1$.

\paragraph{Cartesian product} \label{para:prev-cartesian}
Para.~\ref{para:known} item~(\ref{i:previous:cartesian}) above shows that
cartesian product preserves Fomin
lattices.  Surprisingly, the converse is also true.

\begin{theorem} \label{th:cartesian}
If $L$ is a Fomin lattice with differential degree $r$ and weighting
$w$,
and $L$ is the cartesian product of two lattices
$L = L_1 \times L_2$, then
\begin{enumerate}
\item $L_1$ is a Fomin lattice with some differential degree $r_1$ and
  weighting $w_1$.
\item $L_2$ is a Fomin lattice with some differential degree $r_2$ and
  weighting $w_2$.
\item $r = r_1 + r_2$.
\item For any $x \lessdot y$ in $L_1$ and $z$ in $L_2$,
  $w((x,z) \lessdot (y,z)) = w_1(x \lessdot y)$.
\item For any $w$ in $L_1$ and $x \lessdot y$ in $L_2$,
  $w((v,x) \lessdot (v,y)) = w_2(x \lessdot y)$.
\end{enumerate}
\end{theorem}
\begin{proof}
First we show that for any $x \lessdot y$ in $L_1$ and
any $z_1,z_2 \in L_2$,
\begin{align}
w((x,z_1)\lessdot(y,z_1)) & = w((x,z_2)\lessdot(y,z_2)). \label{eq:left-const}
\end{align}
Because of property~\ref{i:prop-umc},
$w$ is constant on cover-projective pairs of elements of $L$,
so equation (\ref{eq:left-const}) holds if $z_1 \lessdot z_2$.
Since any two elements $z_1, z_2 \in L_2$ can be connected by a finite
chain of upward and downward covers, equation (\ref{eq:left-const}) holds for
any $z_1$ and $z_2$.
Thus the $w_1$ in conclusion (4) exists.

Similarly, for any $w_1,w_2 \in L_1$ and any $x \lessdot y$ in $L_2$
\begin{align}
w((w_1,x)\lessdot(w_1,y)) & = w((w_2,x)\lessdot(w_2,y)). \label{eq:right-const}
\end{align}
Thus the $w_2$ in conclusion (5) exists.

We now show that $r_1$ and $r_2$ exist, that is, that $L_1$ is a Fomin
lattice with differential degree $r_1$ and weighting $w_1$ and
similarly for $L_2$.
Consider any $x \in L_1$ and any $z_1,z_2 \in L_2$.
Equation~(\ref{eq:DP-differential}) for $(x, z_1)$ is:
\begin{align}
  \sum_{(p,q)\mvert (p,q)\lessdot(x,z_1)} w((p,q)\lessdot(x,z_1)) + r
  & = \sum_{(t,u)\mvert (x,z_1)\lessdot(t,u)} w((x,z_1)\lessdot(t,u)) \notag \\
  \sum_{p \mvert p \lessdot x} w((p,z_1)\lessdot(x,z_1)) +
  \sum_{q \mvert q \lessdot z_1} w((x,q)\lessdot(x,z_1)) + r
  & = \sum_{t \mvert x \lessdot t} w((x,z_1) \lessdot (t,z_1)) +
      \sum_{u \mvert z_1 \lessdot u} w((x,z_1) \lessdot (x,u)) \notag \\
  \sum_{p \mvert p \lessdot x} w_1(p \lessdot x) +
  \sum_{q \mvert q \lessdot z_1} w_2(q \lessdot z_1) + r
  & = \sum_{t \mvert x \lessdot t} w_1(x \lessdot t) +
      \sum_{u \mvert z_1 \lessdot u} w_2(z_1 \lessdot u) \label{eq:prod-z1} \\
\interject{Similarly, equation~(\ref{eq:DP-differential}) for $(x, z_2)$ is:}
  \sum_{p \mvert p \lessdot x} w_1(p \lessdot x) +
  \sum_{q \mvert q \lessdot z_2} w_2(q \lessdot z_2) + r
  & = \sum_{t \mvert x \lessdot t} w_1(x \lessdot t) +
      \sum_{u \mvert z_2 \lessdot u} w_2(z_2 \lessdot u) \label{eq:prod-z2} \\
\interject{Subtracting (\ref{eq:prod-z2}) from (\ref{eq:prod-z1}) gives}
  \sum_{q \mvert q \lessdot z_1} w_2(q \lessdot z_1) -
  \sum_{q \mvert q \lessdot z_2} w_2(q \lessdot z_2)
  & = \sum_{u \mvert z_1 \lessdot u} w_2(z_1 \lessdot u) -
      \sum_{u \mvert z_2 \lessdot u} w_2(z_2 \lessdot u) \notag \\
  \sum_{u \mvert z_1 \lessdot u} w_2(z_1 \lessdot u) -
  \sum_{q \mvert q \lessdot z_1} w_2(q \lessdot z_1)
  & = \sum_{u \mvert z_2 \lessdot u} w_2(z_2 \lessdot u) -
      \sum_{q \mvert q \lessdot z_2} w_2(q \lessdot z_2) \notag
\end{align}
This shows that
$\sum_{u \mvert z \lessdot u} w_2(z \lessdot u) -
\sum_{q \mvert q \lessdot z} w_2(q \lessdot z)$ is independent of $z
\in L_1$.
We define $r_2$ as this value.

Similarly, we show that $\sum_{t \mvert z \lessdot t} w_1(z \lessdot t) -
\sum_{q \mvert q \lessdot z} w_1(q \lessdot z)$ is independent of
$z \in L_1$, and we define $r_1$ as this value.

From these facts it is straightforward to prove the conclusions of
the theorem.
\end{proof}

\begin{lemma} \label{lem:cartesian-ass}
If $L$ is a Fomin lattice with differential degree $r$ and weighting
$w$, then
\begin{enumerate}
\item If $L$ has a minimum element $\zerohat_L$, then $L_1$ has a minimum
  element $\zerohat_{L_1}$ and $L_2$ has a minimum
  element $\zerohat_{L_2}$.
\item If $L$ has a minimum element $\zerohat_L$ and is positive,
  and $L_1$ and $L_2$ are not the trivial lattice,
  then $L_1$ and $L_2$ are positive.
\item If $L$ is distributive, then $L_1$ and $L_2$ are distributive.
\end{enumerate}
\end{lemma}
\begin{proof} \disconnect
Regarding (1) and (3): These are straightforward lattice reasoning.

Regarding (2):
By the construction in th.~\ref{th:cartesian},
$L_1$ has some differential degree $r_1$ and weighting $w_1$ and
$L_2$ has differential degree $r_2$ and weighting $w_2$.;
Every value of $w_1$ and every value of $w_2$ are values of $w$ for
suitable arguments, so all values of $w_1$ and $w_2$ are elements of
the set of values $V$ of $L$.
Thus the set of values $V_1$ and
$V_2$ of $L_1$ and $L_2$ can be considered to be $\mbbZ$,
and all values of $w_1$ and $w_2$ are positive.
Since $L$ has a $\zerohat_L$, $L_1$ has a $\zerohat_{L_1}$ and
$L_2$ has a $\zerohat_{L_2}$.
By lem.~\ref{lem:w-positive}, both $r_1$ and $r_2$ are positive.
Thus $L_1$ and $L_2$ are positive.
\end{proof}

\begin{remark} \label{rem:cartesian} \disconnect
(1) A McLarnan map (def.~\forwardref{def:fomin}\forwardref{i:fomin-C})
on $L$ may not factor into McLarnan maps on $L_1$ and $L_2$ in any nice way.

(2) The trivial lattice is a Fomin lattice:  Its $w$ must be the empty
function and $r = 0$.  It is the identity of the cartesian
product of Fomin lattices.
By our definition, this is not a positive weighting.
Thus lem.~\ref{lem:cartesian-ass}(2) requires the hypothesis that
$L_1$ and $L_2$ are not the trivial lattice.%
\footnote{Though our
presentation of the theory might be made more uniform if we defined
the trivial Fomin lattice to be positive.}

(3) Lem.~\ref{lem:cartesian-ass}(2)  requires the hypothesis that
$\zerohat_L$ exists:
Consider $\mbbY$, the lattice of partitions, which is a positive
Fomin lattice
with $w = 1$ and $r=1$.  Consider $\mbbY^*$, the lattice dual of
$\mbbY$, which is a Fomin lattice with $w =1$ and $r=-1$ and so is not
positive.
Define $2\mbbY$ to be the Fomin lattice $\mbbY$ with $w = 2$ and
$r = 2$, which is positive.
Then $L = 2\mbbY \times \mbbY^*$ is a Fomin lattice with
$w = 1 \textup{ or } 2$ and $r = 1$, which is positive.
Thus $L = 2\mbbY \times \mbbY^*$ is a positive Fomin lattice that
is a product of Fomin lattices, one of which is not positive.
\end{remark}

\paragraph{Cartesian product of distributive Fomin lattices} \label{para:prev-cartesian-dist}
The results in this paragraph are only relevant under the assumption
that $L$ is distributive and that its poset of join-irreducibles is
unique-cover-modular.  Hence, they are not significant until
para.~\ref{para:unique-cover-modular}.  However, since these results are not new,
we place them in this section.

\begin{lemma} \label{lem:cartesian-union}
Let $L$ be a finitary%
\footnote{``Finitary'' means that all principal ideals are finite.
For lattices, it is equivalent to $\zerohat$ exists and the lattice is
locally finite.}
distributive lattice with poset of
join-irreducible elements $P = J(L)$.
Then $L$ is the cartesian product of two non-trivial lattices
$L = L_1 \times L_2$ iff
$P$ is the disjoint union of two non-empty posets
$P = P_1 \sqcup P_2$,
in which case $L_1 = \Idealf(P_1)$ and $L_2 = \Idealf(P_2)$.%
\footnote{We define $\Idealf(P)$, where $P$ is a poset, to be the
lattice of finite lower order ideals of $P$.}
\end{lemma}

\begin{lemma} \label{lem:ass-poset}
Let $P$ be a poset and $P = P_1 \sqcup P_2$.
Then
\begin{enumerate}
\item
If $P$ is cover-modular, then $P_1$ and $P_2$ are
cover-modular.  (def.~\ref{def:cover-modular}.)
\item
If $P$ is unique-cover-modular, then $P_1$ and $P_2$ are
unique-cover-modular.  (def.~\ref{def:unique-cover-modular}.)
\end{enumerate}
\end{lemma}

\begin{remark} \label{rem:fuc}
Suppose $L$ is a distributive lattice, $L$ is finitary, and $L$ has
finite upward covers.  It is not necessary that $P$ has finite upward
covers, as is shown by the counterexample shown in fig.~\ref{fig:Fuc}.

\begin{figure}[htp]
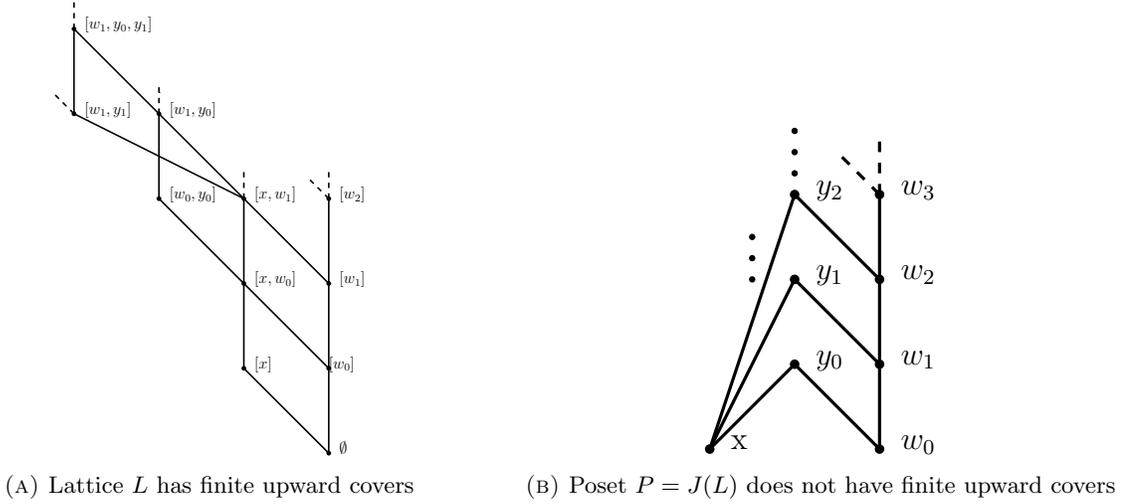

\centering
\begin{subfigure}[b]{0.45\textwidth}
  \centering
  \includegraphics[page=\ipeFigFucL,scale=0.5]{k-row-figs.pdf}
  \caption{Lattice $L$ has finite upward covers}
  \label{fig:FucL}
\end{subfigure}
\begin{subfigure}[b]{0.45\textwidth}
  \centering
  \includegraphics[page=\ipeFigFucJ,scale=1]{k-row-figs.pdf}
  \caption{Poset $P = J(L)$ does not have finite upward covers}
  \label{fig:FucJ}
\end{subfigure}
\caption{Counterexample in rem.~\ref{rem:fuc}}
\label{fig:Fuc}
\end{figure}

Define $P$ to be the poset with elements
$x, y_0, y_1, y_2, \ldots, w_0, w_1, w_2, \ldots$ and the coverings:
\begin{itemize}
\item $w_0 \lessdot w_1 \lessdot w_2 \lessdot \cdots$,
\item $x \lessdot y_i$ \textup{for all} $i \geq 0$, and
\item $w_i \lessdot y_i$ \textup{for all} $i \geq 0$.
\end{itemize}

Define $L$ to be the lattice $L = \Idealf(P)$ of all finite ideals of $P$.  By
Birkhoff's Representation Theorem%
\cite{Birk1967a}*{\S III.3 Th.~3}\linebreak[0]\cite{Stan2012a}*{Prop.~3.4.3}
$L$ is distributive, $L$ is finitary, and $P$ is poset-isomorphic to
the set of join-irreducible elements of $L$.  The
elements of $L$ are
(where $[a, b, c, \ldots]$ is the ideal of $P$ generated by the elements
$a, b, c, \ldots$):

\begin{itemize}
\item $[w_i] = \{w_0, w_1, w_2, \ldots, w_i\}$ \textup{for all} $i \geq -1$
  (generating $\zeroslash$ when $i = -1$),
\item $[x, w_i] = \{x, w_0, w_1, w_2, \ldots, w_i\}$ \textup{for all}
 $i \geq -1$ (generating $\{x\}$ when $i = -1$), and
\item $[w_i, y_{n_1}, y_{n_2}, y_{n_3}, \ldots y_{n_j}] =
\{x, w_0, w_1, w_2, \ldots, w_i, y_{n_1}, y_{n_2}, y_{n_3}, \ldots y_{n_j}\}$
\textup{for all} $i \geq 0$,
\textup{all} $j \geq 1$, and
\textup{all} $0 \leq n_1 < n_2 < n_3 < \cdots < n_j \leq i$
(some non-empty finite set of $y_\bullet$ added to a compatible ideal
of the preceding type),
\end{itemize}

We see that each of the above types of elements of $L$ (ideals of $P$)
has only a finite number of ideals that cover it in $L$, that is,
contain one more element than it does.  So $L$ has finite upward covers.
We see that $x$ has all of the $y_\bullet$ as upward covers in $P$.
So $P$ does not have finite upward covers.
\end{remark}

\section{Fomin lattice essentials} \label{sec:fomin}

The celebrated Robinson--Schensted
algorithm\cite{Rob1938a}\linebreak[0]\cite{Schen1961a}\linebreak[0]%
\cite{Fom1994a}\linebreak[0]\cite{Fom1995a}
and its variants that have attracted substantial attention
can be constructed\cite{Wor2023c}*{sec.~2} from the following
machinery.%
\footnote{Though since a Robinson--Schensted variant can be
constructed from any DGG, it is not clear why DGGs that are not Fomin
lattices have not been as productive.  See
item~(\forwardref{i:why-fomin-lattices}) in sec.~\ref{sec:future}.}
The core of the machinery is a ``Fomin lattice'' which satisfies
properties~\ref{i:prop-DGG}--\ref{i:prop-mod-lattice} of
para.~\ref{para:properties}, summarized here:

\begin{definition} \label{def:fomin}
\disconnect
\begin{enumerate}
\itemlabel{(A)}{i:fomin-A} A \emph{Fomin lattice},%
\footnote{We call it a Fomin lattice because all of the required
properties are explicitly or implicitly given in \cite{Fom1994a} and
\cite{Fom1995a}.}
which comprises:
\begin{enumerate}
\itemlabel{(1)}{i:fomin-A1} a lattice $L$ that
\begin{enumerate}
\itemlabel{(a)}{i:fomin-A1a} is modular,
\itemlabel{(b)}{i:fomin-A1b} is locally finite, and
\itemlabel{(c)}{i:fomin-A1c} has finite upward and downward covers
  for each element;
\end{enumerate}
\itemlabel{(2)}{i:fomin-A2} a \emph{differential degree} $r$ which
  is a member of a set of \emph{values} $V$;
\itemlabel{(3)}{i:fomin-A3} a \emph{weighting}
  $w(\bullet \lessdot \bullet)$, which is a
  function from \emph{coverings}
  (or \emph{prime quotients}), pairs $x \lessdot y$ of elements of $L$, to
  $V$, for which $L$ is a \emph{weighted-differential
  lattice}\cite{Stan1990a}*{sec.~3}\cite{Fom1994a}*{sec~2.2}:
\begin{equation}
  \sum_{y\mvert y \lessdot x} w(y \lessdot x) + r =
  \sum_{z\mvert x \lessdot z} w(x \lessdot z)
  \textup{\qquad\qquad for all } x \in L
\label{eq:differential}
\end{equation}
(which is often described as \emph{$r$-differential}); and
\itemlabel{(4)}{i:fomin-A4} for which $w$ is projective-constant,
  that is, constant on projectivity classes\cite{Birk1967a}*{\S I.7 Def.}
  of coverings, or equivalently, on coverings which are
  \emph{cover-projective}:
  if $x \vee y$ covers $x$ and $y$ and $x \wedge y$ is covered by $x$ and $y$,
  then $w(x \wedge y \lessdot x) = w(y \lessdot x \vee y)$
  and $w(x \wedge y \lessdot y) = w(x \lessdot x \vee y)$.
\end{enumerate}
\itemlabel{(B)}{i:fomin-B} The lattice $L$ has a minimum element, $\zerohat$.
\itemlabel{(C)}{i:fomin-C} A \emph{McLarnan map}%
\footnote{We name the correspondence for McLarnan, who seems to have
been the first to emphasize\cite{McLar1986a}*{Sec.~1 p.~22} the arbitrariness
of the choice of the correspondence and thus that there is no
``natural'' correspondence.  The set of possible correspondences for a
particular Fomin lattice is
usually uncountable.  We choose ``map'' because of its
consonance with ``McLarnan''.}
or \emph{$R$-correspondence} $M$, which
  is a realization of the weighted-differential condition:
  \[ \textup{for all $x \in L$,
  $M_x$ is a bijection between
  $\{y \in L\mvert y \lessdot x\} \sqcup \{1,2,3,\ldots r\}$ and
  $\{z \in L\mvert x \lessdot z\}$}.
	\]
\itemlabel{(D)}{i:fomin-D}
Depending on the degree of abstraction of the discussion, the
requirements for the set of values $V$ varies.
The minimum requirement is that $V$ be a module over a unitial ring of
\emph{scalars} $S$.  Usually we will be unconcerned with the details of
$S$ other than it necessarily contains a homomorphic image of $\mbbZ$
and thus we can multiply elements of $V$ by integer coefficients.
If $V = \mbbZ$ (which implies $\mbbZ$ is \emph{isomorphically}
embedded in $S$), $r > 0$, and all of the values of $w$ are positive,
then the
weighting is defined to be \emph{positive}, which is required to
construct a Robinson--Schensted algorithm.
\end{enumerate}
\end{definition}

\begin{lemma} \label{lem:linear-combos}
If a lattice $L$ forms a Fomin lattice with two pairs of
differential degrees and
weightings $(r, w)$ and $(r^\prime, w^\prime)$, then $L$ forms a Fomin
lattice with any linear combination of the two:
the differential degree
$r^{\prime\prime} = \alpha r + \beta r^\prime$ and weighting
$w^{\prime\prime}(x \lessdot y) = \alpha w(x \lessdot y) + \beta
w^\prime(x \lessdot y)$ for any $\alpha, \beta \in S$.  Thus, the set of
pairs of degrees and weightings for $L$ forms a
subspace of an infinite-dimensional module over $S$,
the subspace of solutions of the set of
linear equations~(\ref{eq:differential}).

If $(r,w)$ and $(r^\prime,w^\prime)$ are positive,
$\alpha \geq 0$, $\beta \geq 0$, and not both $\alpha$ and $\beta$ are 0,
then $(r^{\prime\prime}, w^{\prime\prime})$ is positive.
\end{lemma}

\begin{lemma}
Because $L$ is modular and $w$ satisfies the projectivity
condition \ref{i:fomin-A}\ref{i:fomin-A4},
the relevant version of the Jordan--H\"older
theorem\cite{Birk1967a}*{\S III.7 Th.~9 Cor.}\linebreak[0]%
\cite{Nat2017a}*{Th.~9.7} shows that $L$ is graded and
that the multiset of the weights of the coverings composing a
path from some $x$ to some $y \geq x$ in $L$ depends only on $x$ and
$y$ and not on the particular path.
\end{lemma}

\begin{definition}
We define the \emph{number of colorings} of $x \in L$, denoted by
$c(x)$,
to be the product of the weights of the coverings composing an
arbitrary path from $\zerohat$ to $x$ in $L$.
\end{definition}

\begin{theorem}
Given this machinery, the general Robinson--Schensted insertion
algorithm can be constructed using Fomin's growth diagram
construction provided the weighting is positive.
The set of permutations of $\{1, \ldots, n\}$ whose elements are
$r$-colored can be bijectively mapped to the set
\begin{align*}
\{ (P, Q, c) \mvert
& \textup{there exists $x \in L$ of grade $n$ for which } \\
& P \textup{ is a path in $L$ from $\zerohat$ to $x$ and } \\
& Q \textup{ is a path in $L$ from $\zerohat$ to $x$ and } \\
& c \textup{ maps each covering $y \lessdot z$ in $P$ to an element of }
  \{1, 2, 3, \ldots, w(y \lessdot z)\} \}
\end{align*}
where the particular bijection depends on the McLarnan map
that is chosen.%
\footnote{We use ``path'' rather than the more usual ``saturated chain'' to
emphasize the importance of the coverings composed of adjacent elements of
the chain, which bear the weights.}
\end{theorem}

\begin{definition}
Each element $x$ of $L$ with grade $n$ is called a \emph{diagram} of
\emph{size} $n$.  Each path from $\zerohat$ to $x$ is called
a \emph{tableau} of \emph{shape} $x$ and thus size $n$.
We define $f(x)$ as the number
of paths from $\zerohat$ to $x$, that is, the number of tableaux of
shape $x$.
\end{definition}

\begin{theorem}
The most celebrated consequence of the general insertion algorithm is
the immediate consequence of the above bijection:
\begin{equation}
n!\,r^n = \sum_{\textup{diagrams $x$ with size $n$}} f(x)^2 c(x).
\label{eq:n-fact}
\end{equation}
\end{theorem}

\section{Auxiliary result} \label{sec:aux}

The result here supports the discussion of future
directions (sec.~\ref{sec:future}).
It is placed here to avoid interacting with the
assumptions used in sec.~\ref{sec:local} and~\ref{sec:global}.

\paragraph{Partitioning a lattice} \label{para:partition}

Let $L$ be a locally finite modular lattice and let $S$ be a set of
coverings in $L$ that is closed under cover-projectivity, that is, if
$x \lessdot y$ is in $S$ and it is cover-projective to
$x^\prime \lessdot y^\prime$, then $x^\prime \lessdot y^\prime$ is in $S$.
Since $L$ is locally finite and modular, this is equivalent to that
$S$ is the union of a set of projectivity classes of coverings in $L$.

We will define a relationship $\sim$ between elements of $L$.
Intuitively, $x \sim y$ if there is a (finite) path between $x$ and
$y$ composed of (upward and downward) covering steps, none of which
are in $S$.  But it's easier to analyze $\sim$ by defining it in a
different way.

\begin{lemma} \label{lem:sim-path}
Given $x \leq y$ in $L$ and any two paths between them, one path
contains a covering in $S$ iff the other path contains a covering in $S$.
\end{lemma}
\begin{proof} \disconnect
Since $L$ is modular and locally finite, the interval $[x,y]$ is modular
and finite.  Choose any projectivity class $C$ of coverings in $L$.
The relevant version of the Jordan--H\"older
theorem\cite{Birk1967a}*{\S III.7 Th.~9 Cor.}\linebreak[0]%
\cite{Nat2017a}*{Th.~9.7} can be applied $[x,y]$ to show that
one path contains a covering in $C$ iff the other path contains a
covering in $C$.  The lemma then follows.
\end{proof}

Thus the following definition does not depend on the particular path
chosen between $x \wedge y$ and $x \vee y$.

\begin{definition} \label{def:sim}
We define $x \sim_S y$ (or $x \sim y$ if, as usual, $S$ is understood)
iff no covering in a path from $x \wedge y$ to
$x \vee y$ is in $S$.
\end{definition}

\begin{lemma} \label{lem:sim-leq}
If $x \leq y$ in $L$, then $x \sim y$ iff no covering in a path
from $x$ to $y$ is in $S$.
\end{lemma}

\begin{lemma} \label{lem:sim-convex}
If $x \leq y \leq z$ in $L$ and $x \sim z$, then $x \sim y$ and
$y \sim z$.
\end{lemma}
\begin{proof} \disconnect
By lem.~\ref{lem:sim-leq}, no covering in any path from $x$ to $z$ is
in $S$.  We can choose such a path
from $x$ to $y$ and from there to $z$.  Since no
covering in that path is in $S$, no covering from $x$ to $y$ or from
$y$ to $z$ is in $S$.  Then by lem.~\ref{lem:sim-leq}, $x \sim y$
and $y \sim z$.
\end{proof}

\begin{lemma} \label{lem:sim-transitive}
Conversely, if $x \leq y \leq z$ in $L$, $x \sim y$, and
$y \sim z$, then $x \sim y$.
\end{lemma}

\begin{lemma} \label{lem:sim-diamond}
If $x, y \in L$, then $x \sim x \vee y$ iff $x \wedge y \sim y$.
\end{lemma}
\begin{proof} \disconnect
This follows from the diamond isomorphism theorem and
lem.~\ref{lem:sim-leq}.
\end{proof}

\begin{lemma} \label{lem:sim-ud}
If $x, y \in L$, $x \sim y$ implies
$x \wedge y \sim x \vee y$,
$x \wedge y \sim x \sim x \vee y$, and
$x \wedge y \sim y \sim x \vee y$.
\end{lemma}

\begin{lemma} \label{lem:sim-ud-inverse}
If $x, y \in L$, $x \sim x \vee y \sim y$ implies $x \sim y$.
Dually, $x \sim x \wedge y \sim y$ implies $x \sim y$.
\end{lemma}
\begin{proof} \disconnect
Given $x \sim x \vee y$, by lem.~\ref{lem:sim-diamond}, $x \wedge y \sim y$.
Since $x \wedge y \leq y \leq x \vee y$, $x \wedge y \sim y$,
and $y \sim x \vee y$,
By lem.~\ref{lem:sim-transitive}, $x \wedge y \sim x \vee y$ and so $x \sim y$.
The dual is proved dually.
\end{proof}

\begin{lemma} \label{lem:sim-equiv}
Given $x, y, z \in L$, if $x \sim y$ and $y \sim z$, then $x \sim z$.
\end{lemma}

\begin{proof} \disconnect

The construction for this proof is shown in fig.~\ref{fig:SimEquiv}.

\begin{figure}[htp]
\centering
\includegraphics[page=\ipeFigSimEquiv,scale=0.66]{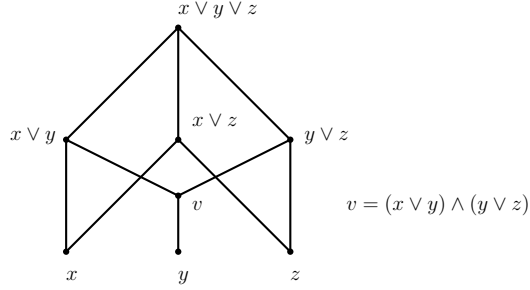}
\caption{Construction for the proof of lem.~\ref{lem:sim-equiv}.}
\label{fig:SimEquiv}
\end{figure}

We know $x \sim y$ and $y \sim z$.  By lem.~\ref{lem:sim-ud},
$x \sim x \vee y$, $x \vee y \sim y$, $y \sim y \vee z$,
and $y \vee z \sim z$.
Define $v = (x \vee y) \wedge (y \vee z)$.
Of necessity, $y \leq v \leq x \vee y$ and
$y \leq v \leq y \vee z$.
By lem.~\ref{lem:sim-convex}, $v \sim x \vee y$ and $v \sim y \vee z$.
By lem.~\ref{lem:sim-ud-inverse}, $x \vee y \sim y \vee z$
and by lem.~\ref{lem:sim-ud},
$x \vee y \sim (x \vee y) \vee (y \vee z) = x \vee y \vee z$.
Since $x \leq x \vee y \leq x \vee y \vee z$,
by lem.~\ref{lem:sim-transitive}, $x \sim x \vee y \vee z$.
Since $x \leq x \vee z \leq x \vee y \vee z$,
by lem.~\ref{lem:sim-convex}, $x \sim x \vee z$.
Similarly, $z \sim x \vee z$.
Then by lem.~\ref{lem:sim-ud-inverse}, $x \sim z$.
\end{proof}

\begin{lemma} \label{lem:sim-equivalence}
$\sim$ is an equivalence relation on the elements of $L$.
\end{lemma}
\begin{proof} \disconnect
That $\sim$ is reflexive and symmetric are trivial.
Lem.~\ref{lem:sim-equiv} shows that $\sim$ is transitive.
\end{proof}

\begin{lemma} \label{lem:sim-meet-join}
Given $x, y, z \in L$, if $x \sim y$ and $x \sim z$, then
$x \sim y \wedge z$ and $x \sim y \vee z$.
\end{lemma}
\begin{proof} \disconnect
By lem.~\ref{lem:sim-equiv}, $y \sim z$, and by lem.~\ref{lem:sim-ud},
$y \sim y \wedge z$.  Then by lem.~\ref{lem:sim-equiv},
$x \sim y \wedge z$.  Dually, $x \sim y \vee z$.
\end{proof}

\begin{lemma} \label{lem:sim-convex-4}
Given $x, y, z, v \in L$ with $y \leq z \leq v$,
if $x \sim y$ and $x \sim v$, then $x \sim z$.
\end{lemma}
\begin{proof} \disconnect
By lem.~\ref{lem:sim-equiv}, $y \sim v$, and by
lem.~\ref{lem:sim-convex}, $y \sim z$.
Then by lem.~\ref{lem:sim-equiv}, $x \sim z$.
\end{proof}

Now we fix an element $z \in L$.  $z$ will usually be $\zerohat$.

\begin{definition} \label{def:truncate}
We define $L_{/Sz}$ to be the equivalence class of $\sim_S$ containing
$z$, which is $\{ x \in L \mvert x \sim_S z \}$.
\end{definition}

Intuitively, we construct $L_{/Sz}$ by constructing the Hasse diagram
of $L$, removing from it the edges which are coverings in $S$,
selecting from the resulting graph the connected component that contains
$z$, and considering that component as the Hasse diagram of a lattice,
which is $L_{/Sz}$.

\begin{lemma} \label{lem:sim-conv-subl}
$L_{/Sz}$ is a convex sublattice of $L$.
\end{lemma}
\begin{proof} \disconnect
That $L_{/Sz}$ is convex is shown by lem.~\ref{lem:sim-convex-4}.
That $L_{/Sz}$ is a sublattice is shown by lem.~\ref{lem:sim-meet-join}.
\end{proof}

We now apply this construction to a Fomin lattice that has a
non-positive weighting.

\begin{definition} \label{def:truncate-fomin}
Given a Fomin lattice $L$ with weighting $w$ and $\zerohat$, we define its
\emph{truncation} $L_/$ to be
$L_{/S\zerohat}$, where $S = \{ x \lessdot y \mvert w(x \lessdot y) = 0 \}$,
the set of coverings with weight 0.
We define the weighting $w_/$ on $L_/$ to be $w$ restricted to $L_/$.
\end{definition}

\begin{theorem} \label{th:truncate-fomin}
Given a Fomin lattice $L$ with $\zerohat$, its truncation $L_/$ is a
Fomin lattice with (the same) $\zerohat$.
\end{theorem}
\begin{proof} \disconnect
All of the requirements for $L_/$ to be a Fomin lattice
(def.~\ref{def:fomin}) are trivially inherited from $L$ except for
\ref{i:fomin-A3}.
But by construction, equation (\ref{eq:differential}) for a particular
$x \in L_/$ contains a subset of the terms of (\ref{eq:differential})
for $x \in L$, and the terms of the latter that are not in the former
are ones for which $w$ has the value 0.
Thus (\ref{eq:differential}) is true for all $x \in L_/$.
\end{proof}

\begin{theorem} \label{th:truncate-fomin-dist}
Given a finitary distributive Fomin lattice $L$ with $\zerohat$ with points
$J(L)$, its truncation $L_/$ is isomorphic to $\Idealf(P_/)$
where $P_/$ is the subposet
$\{ j \in J(L) \mvert
\textup{not } (\exists k)(k \leq j \textup{ and } w(j) = 0) \}$.
\end{theorem}

\begin{theorem} \label{th:truncate-fomin-admis}
Given a Fomin lattice $L$ with $\zerohat$, if
its set of values is $\mbbZ$, $r > 0$, and for every covering
$x \lessdot y$ for which $w(x \lessdot y) < 0$, there exists
a $z, v$ for which $z \lessdot v \leq x$ and $w(z \lessdot v) = 0$,
then $L_/$ is a Fomin lattice with $\zerohat$ and a positive weighting.
\end{theorem}
\begin{proof} \disconnect
By th.~\ref{th:truncate-fomin}, $L_/$ is a Fomin lattice with
$\zerohat$.  It remains to prove that its weighting $w_/$ is
has only positive values on $L_/$.  Assume to the contrary that
$x \lessdot y$ in $L_/$ and $w_/(x \lessdot y) \leq 0$.
If $w_/(x \lessdot y) = 0$, then by the construction of $L_/$ and
lem.~\ref{lem:sim-leq}, $y \not\in L_/$.
If on the other hand $w_/(x \lessdot y) < 0$, then by hypothesis there
is $\zerohat \leq z \lessdot v \leq y$ with $w_/(z \lessdot v) = 0$,
and similarly $v \not\in L_/$, implying $y \not\in Z_/$.
\end{proof}

\begin{theorem} \label{th:truncate-fomin-dist-admis}
Given a distributive Fomin lattice $L$ with $\zerohat$ with points
$P(L)$, if for every point $j$ for which $w(j) < 0$ then
there exists a point $k$ for which $k \leq j$ and $w(k) = 0$,
then $L_/$ is a Fomin lattice with a positive weighting.
\end{theorem}

Of course, this construction of a smaller lattice that is an
positive Fomin lattice could be done by choosing the component
containing a different element $z \in L$ than $\zerohat$,
since if $L$ has $zerohat$, then any sublattice of
$L$ has a minimum element.

The only known application of truncation is when $L$ is $\mbbY$,
$P = J(L)$ is the quadrant, and the weighting on $P$ is
$w((i,j) = \alpha + \beta i + \gamma j$.\cite{Fom1994a}*{Exam.~2.2.7}
If $\alpha = r$, $\beta = r/k$, and $\gamma = -r/k$ for some nonzero
integer $k$ that divides $r$, then $L_/$ is $\mbbY_k$ with
the weights multiplied by $r/k$
(if $k$ is positive) or its left-right reversal (if $k$ is negative).

\section{Local characterization} \label{sec:local}

We start with proving various ``local'' facts about the poset $P$ of
points (join-irreducible elements) of $L$.  These culminate with a
classification of the possible structures of the neighborhood of a
point $p$, that is, the convex set generated by the points that cover
of $p$ and the points that are covered by $p$.

Henceforth we will characterize positive Fomin lattices under
certain sets of assumptions.  As we proceed through the analysis,
further assumptions will be added.  All of these assumptions may seem an
excessively narrow approach, but all known Fomin lattices with
$\zerohat$ other
than $\mbbYF_r$ satisfy all of these assumptions.

\paragraph{Distributive lattice}
\begin{assumption} \label{ass:distributive}
Henceforth, we will restrict our attention to
distributive Fomin lattices with $\zerohat$.
Specifically, we assume $L$ is a Fomin lattice with $\zerohat$
with differential degree $r$, values $V$, and weighting $w$,
and that $L$ is distributive.
(We do not, at this point, assume that $w$ is positive.)
\end{assumption}

\begin{theorem} \label{th:differential-distributive}
\cite{Birk1967a}*{\S III.3 Th.~3}\cite{Stan2012a}*{Prop. 3.4.3}
Since we have assumed that $\zerohat$ exists,
$L$ is finitary, that is, all if its principal ideals are finite.
$L$ is
isomorphic to the set of finite order ideals of the poset
$P = J(L)$ of its join-irreducible elements,
which we denote $L = \Idealf(P)$.
Given a covering $x \lessdot y$
in $L$, $x \setminus y$ (considered as the set difference of two ideals
of $P$) contains exactly one element of $P$, and the projectivity
requirement (condition \ref{i:fomin-A}\ref{i:fomin-A4}) on $w$ can be
simplified to:
There is a weighting function $w$ from $P$ into
$V$, and for every $x \lessdot y$ in $L$,
$w(x \lessdot y) = w(y \setminus x)$.

The weighted-differential condition (\ref{eq:differential}) is
equivalent to
\begin{equation}
  \sum_{a\mvert a \textup{ maximal in } x} w(a) + r =
  \sum_{b\mvert b \textup{ minimal in } P \setminus x} w(b)
  \textup{\qquad\qquad for all $x \in \Idealf(P)$}.
\label{eq:differential-distributive}
\end{equation}
\end{theorem}

Because the image of $w$ in $L$ is the same as
the image of $w$ on $P$, we abuse language by conflating the
two meanings of $w$ when we are dealing them as entire functions.  In
particular, each of them has all of its values positive integers iff the other
does, and so ``$w$ is positive'' is used to mean both.

\begin{definition} \label{def:point}
The elements of $P$ are called \emph{points} or \emph{boxes}.
\end{definition}

\begin{lemma} \label{lem:P-finite} \disconnect
\begin{enumerate}
\item $P$ is locally finite.
\item $P$ is finitary.
\end{enumerate}
\end{lemma}
\begin{proof} \disconnect
Regarding (1):  If there was some $a \leq b$ in $P$ for which $[a,b]$
is infinite, then the set of all (closed) principal ideals of the elements in
$[a,b]$, $\{ [x] \mvert a \leq x \leq b \}$, would be an infinite
subset of $L$ with all of its elements $\geq [a]$ and $\leq [b]$.
That contradicts that $L$ is locally finite.

Regarding (2):  Similarly, if there was an $a \in P$ whose principal
ideal $[a]$ in $P$ was infinite, then
$\{ [x] \mvert x \leq b \}$ would be an infinite
subset of $L$ with all of its elements $\leq [a]$.
That contradicts that, given our assumptions,
th.~\ref{th:differential-distributive} shows that $L$ is finitary.
\end{proof}

\begin{lemma} \label{lem:reduce-to-points}
The lattice of diagrams is isomorphic to the lattice of finite ideals of
points under union and intersection, the size of a diagram is
the cardinality of its corresponding ideal (the number of points it
contains), and the number of tableaux with a
particular shape is the number of linear orderings of its points
(a subposet of $P$).  The number of colorings $c(x)$ of
a diagram $x$
is the product of the weights $w$ of its points:
$c(x) = \prod_{a \in x} w(a)$.
\end{lemma}

\begin{definition} \label{def:ins-del}
Given any diagram $x$, its \emph{insertion points}
are the points that are the minimal elements of $P \setminus x$,
those which can be added to $x$ to make a diagram (whose size is
1 larger than the size of $x$).
The \emph{deletion points} of $x$
are the points that are the maximal elements of $x$,
those which can be deleted from $x$ to make a diagram (whose size is 1
smaller than the size of $x$).
For any diagram $x$,
we define $D_x$ to be the set of deletion points of $x$,
and $I_x$ to be the set of insertion points
of $x$.
\end{definition}

The summation on the left of (\ref{eq:differential-distributive})
is over the deletion points of $x$ and the summation on the right of
(\ref{eq:differential-distributive}) is over the insertion points of
$x$.
Thus, (\ref{eq:differential-distributive}) is equivalent to
\begin{equation}
  \sum_{a \in D_x} w(a) + r =
  \sum_{b \in I_x} w(b)
  \textup{\qquad\qquad for all $x \in \Idealf(P)$}.
\label{eq:differential-distributive-I-D}
\end{equation}

Applying (\ref{eq:differential-distributive}) to the ideal $\zeroslash$
gives
\begin{equation}
  r =
  \sum_{z \mvert z \textup{ minimal in } P} w(z)
\label{eq:minimal}
\end{equation}
In many cases $P$ has a unique minimal element, $\zerohat_P$,
and in that case, (\ref{eq:minimal}) reduces to $w(\zerohat_P) = r$.

Given an ideal $x$ and a point $p$ in $I_x$,
that is, minimal in
$P \setminus x$, we can construct the ideal ${x^\prime} = x \sqcup \{p\}$.
Thus $P \setminus x = (P\setminus x^\prime) \sqcup \{p\}$ and
$p \in D(x^\prime)$.

The difference between (\ref{eq:differential-distributive-I-D}) applied to
$x^\prime$ and (\ref{eq:differential-distributive-I-D}) applied to
$x$ is
\begin{equation}
  \sum_{a \in D_{x^\prime} \setminus D_x} w(a) - \sum_{b \in D_x \setminus D_{x^\prime}} w(b)
= \sum_{a \in I_{x^\prime} \setminus I_x} w(a) - \sum_{b \in I_x \setminus I_{x^\prime}} w(b)
  \textup{\qquad\qquad for all $x \in \Idealf(P)$, $p \in I_x$}.
\label{eq:diff-4-sets}
\end{equation}

The domains of the four summations of (\ref{eq:diff-4-sets}) are:
\begin{enumerate}
\item
$D_{x^\prime} \setminus D_x$: The set of maximal elements of
${x^\prime}$ that are not maximal elements of $x$,
which is $\{p\}$,
\item
$D_x \setminus D_{x^\prime}$: The set of maximal elements of
$x$ that are not maximal elements of ${x^\prime}$,
\item
$I_{x^\prime} \setminus I_x$: The set of minimal elements of
$P \setminus {x^\prime}$ that are not minimal elements of
$P \setminus x$,
\item
$I_x \setminus I_{x^\prime}$: The set of minimal elements of
$P \setminus x$ that are not minimal elements of $P \setminus {x^\prime}$,
which is $\{p\}$.
\end{enumerate}

\begin{lemma} \label{lem:equiv-1}
  Condition (\ref{eq:differential-distributive}) (for every
  $x \in \Idealf(P)$) is equivalent to
  the combination of (\ref{eq:minimal})
  and (\ref{eq:diff-4-sets}) (for every
  $x \in \Idealf(P)$ and $p \in I_x$).
\end{lemma}
\begin{proof} \disconnect
Regarding $\Rightarrow$:  This has been proved by the above derivation of
(\ref{eq:minimal}) and (\ref{eq:diff-4-sets}) from
(\ref{eq:differential-distributive}).

Regarding $\Leftarrow$:
For any $y \in \Idealf(P)$,
we want to prove (\ref{eq:differential-distributive}) for its $x = y$.
Since $y$ is finite, we can prove this by induction upward on the size of $y$.
The base case is $y = \zeroslash = \zerohat$.
In that case, (\ref{eq:differential-distributive}) with
its $x = \zerohat$ reduces to (\ref{eq:minimal}).

If $y$ has positive size, let $y^\prime \lessdot y$ in $\Idealf(P)$.
By induction, (\ref{eq:differential-distributive}) is
true for its $x = y^\prime$.
Now consider (\ref{eq:diff-4-sets}) with its $x = y^\prime$ and
its $p$ the unique element of $y \setminus y^\prime$.
We can add these two equations and by the above derivation obtain
(\ref{eq:differential-distributive}) for its $x = y$.
\end{proof}

\begin{definition} \label{def:CNS}
Given a point $p \in P$,
define $C^-_p$
($C^-$ for ``downward covers'') as
the set of all elements of $P$ which are covered by $p$.
Similarly,
define $C^+_p$ ($C^+$ for ``upward covers'') as the
set of all elements of $P$ which cover $p$.
Define $N_p$ ($N$ for ``neighborhood'') to be
$\{ x \in P \mvert (\exists a\in C_p^-, b\in C_p^+)\, a \leq x \leq b \}$,
the convex subset of $P$ generated by $C_p^-$ and $C_p^+$.
Define $S_p$ ($S$ for ``siblings'') to be
$\{ x \in P \mvert x \neq p \textup{ and }
(\exists a\in C_p^-, b\in C_p^+)\, a < x < b \} =
N_p \setminus C_p^- \setminus C_p^+ \setminus \{p\}$,
$N_p$ with $C_p^-$, $C_p^+$, and $p$ deleted.
\end{definition}

\begin{lemma}
By construction, every element of $C_p^+$ is $>$ every element of $C_p^-$,
but no element of $C_p^+$ covers any element of $C_p^-$.
All elements of $C_p^+$ are incomparable to each other.
Dually for $C_p^-$.
Every element of $S_p$ is incomparable to $p$.
If $P$ is graded, every element of $S_p$ has the same grade as $p$ and
$S_p$ is an antichain.  However, in general, $S_p$ need not be an antichain.
\end{lemma}

\begin{lemma}
Given an ideal $x$
and a point $p$ that is an insertion point of $x$,
$C_p^- \subset x$ and $C_p^+ \subset P\setminus (x \sqcup \{p\})$.
\end{lemma}

\begin{definition} \label{def:Spx}
Given an ideal $x$ and a point $p$ in $I_x$,
define $S^-_{px} = S_p \cap x$ and
$S^+_{px} = S_p \cap (P\setminus (x \sqcup \{p\}))$.
\end{definition}

\begin{lemma} \label{lem:Spx-ideal}
Since $p \not\in S$, by construction, $S_p = S_{px}^- \sqcup S_{px}^+$, and if
$s^- \in S_{px}^-$ and $s^+ \in S_{px}^+$ are comparable, then $s^- \leq s^+$.
So $S_{px}^-$ is an ideal of $S_p$ and $S_{px}^+$ is a filter of $S_p$.
\end{lemma}

The generic relationships of the elements of these sets are shown in
fig.~\ref{fig:JXB}.
\begin{figure}[htp]
\centering
\includegraphics[page=\ipeFigJXB,scale=0.66]{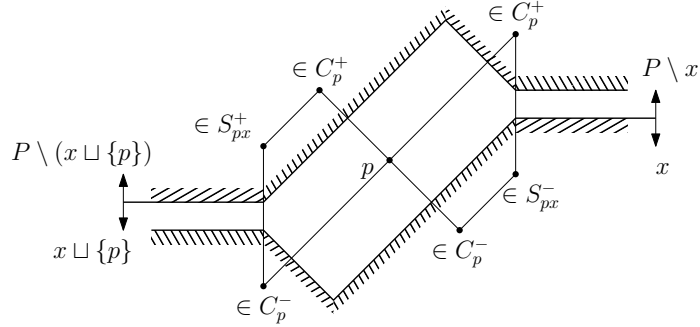}
\caption{Adding a point $p$ to an ideal $x$ making an ideal $x^\prime$}
\label{fig:JXB}
\end{figure}

\paragraph{Unique-cover-modular} \label{para:unique-cover-modular}

\begin{assumption} \label{ass:unique-cover-modular}
Henceforth, we assume the poset
of points is unique-cover-modular (def.~\ref{def:cover-modular}).
\end{assumption}

The following example disproves a conjecture that was in preliminary
versions of this article.
\begin{remark}
The poset $t$-\textit{Plait} in \cite{Fom1994a}*{Exam.~2.2.10 and Fig.~7}
is locally finite and cover-modular but is not graded.%
\footnote{``Graded'' means that every saturated chain from $a$ to $b$ has
the same length, but there need not be a $\zerohat$.  In general, the
grade function of a graded poset ranges over the integers, not just the
nonnegative integers.}
\end{remark}

\begin{lemma} \label{lem:x-plus}
Let $p \in P$.
For any $a \in C_p^+$ and any $b \neq p$ that is covered by $a$,
unique-cover-modularity shows that there is a $c \in C_p^-$ that is
covered by $a$ and so $b \in S_p$.
Dually, for any $a \in C_p^-$ and any $b \neq p$ that covers $a$,
there is a $c \in C_p^+$ that covers $a$ and so $b \in S_p$.
For any $a \in S_p$, unique-cover-modularity shows there
is a unique element of $C_p^+$ which is $> a$.
Dually, there is a unique element of $C_p^-$ which is $< a$.
\end{lemma}

\begin{definition} \label{def:x-plus}
For any $p \in P$ and any $a \in S_p$, we define
$a^+_p$ to be the unique element of $C_p^+$ which is $> a$.
Dually, we define $a^-_p$ to be the unique element of $C_p^-$ which is $< a$.
\end{definition}

\begin{lemma} \label{lem:equiv-2}
The four summation domain sets of (\ref{eq:diff-4-sets}) are:
\begin{enumerate}
\item $D_{x^\prime} \setminus D_x = \{p\}$
\item $D_x \setminus D_{x^\prime}$
is the set of maximal elements of $x$ which are covered by $p$.
This is the same as the elements of $C_p^-$ which are not $\leq$ (or
equivalently, not covered by) any element of $S_{px}^-$.
We denote this property by $\not\leq S_{px}^-$.
\item $I_{x^\prime} \setminus I_x$
is the set of minimal elements of $P\setminus x^\prime$ which cover $p$.
This is the same as the elements of $C_p^+$ which are not $\geq$ (or
equivalently, do not cover) any element of $S_{px}^+$.
We denote this property by $\not\geq S_{px}^+$.
\item $I_x \setminus I_{x^\prime} = \{p\}$
\end{enumerate}
This shows (\ref{eq:diff-4-sets}) is equivalent to
\begin{align}
  w(p) - \sum_{a \mvert a \in C_p^-,\  a \not\leq S_{px}^-} w(a)
  & = \sum_{b \mvert b \in C_p^+,\ b \not\geq S_{px}^+} w(b) - w(p)
& \textup{for all $x \in \Idealf(P)$, $p \in I_x$}
  \notag \\
\noalign{or equivalently}	
\sum_{a \mvert a \in C_p^-,\ a \not\leq S^-_{px}} w(a) +
\sum_{b \mvert b \in C_p^+,\ b \not\geq S^+_{px}} w(b)
& = 2 w(p)
& \textup{for all $x \in \Idealf(P)$, $p \in I_x$}.
\label{eq:diff-2-sets}
\end{align}
\end{lemma}
\begin{proof} \disconnect
Regarding (1) and (4):  These are straightforward.

Regarding (2) $\Rightarrow$:
Let $a \in D_x \setminus D_{x^\prime}$.
Then $a$ is maximal in $x$ but not maximal in ${x^\prime}$.
Thus $a \leq$ an element of $x^\prime$ which is not an element of $x$.
The only  element of $x^\prime$ that is not an element of $s$ is $p$, so $a < p$.
If $a$ is not covered by $p$, then there is a $b$ with $a < b < p$.
Thus $b \in x^\prime$ but since $a$ is maximal in $x$,
$b \not\in x$, which requires that $b = p$,
which is a contradiction. Thus $a$ is covered by $p$ and so $a \in C^-_p$.

To show that $a \not\leq S_{px}^-$, assume there is a $c$
in $S_{px}^-$ with $a \leq c$. Since $C^-_p$ is disjoint from
$S^-_{ps}$, $a < c$.  But $c \in x$, contradicting
that $a$ is maximal in $x$.  Thus $a \not\leq S_{px}^-$.

Regarding (2) $\Leftarrow$:
Let $a \in C^-_p$ and $a \not\leq S^-_{px}$.
Because $a < p$ and $p \in x^\prime$, $a$ is not maximal in $x^\prime$.

To show that $a$ is maximal in $x$, assume there is a $d$ in $x$ with
$d > a$.  We can assume that $d \gtrdot a$.
Then we apply
unique-cover-modularity to $p$ and $d$ to obtain a $e$ which covers
$p$ and $d$.  So $e \in C^+_p$, and with $a \in C^-_p$ shows
$d \in S_p$.  Since $d \in x$, $d \in S^-_{px}$.
This contradicts that $a \not\leq S^-_{ps}$, so $a$ is maximal in $x$.

Regarding (3): This is proved dually to (2).
\end{proof}

\paragraph{Reorganizing the linear equations}
Looking at (\ref{eq:diff-2-sets}) as a set of linear equations in the values
of $w$, we see that each equation is determined by choosing
first an $x \in \Idealf(P)$ and then a $p \in I_x$.
We will now reorganize the statement of the set of linear equations
(\ref{eq:diff-2-sets}),
but indexed first by an arbitrary choice
of $p \in P$ (which determines $C_p^+$ and $C_p^-$), and then
by the values of
$S_{px}^+$ and $S_{px}^-$ consistent with $p$.

\begin{lemma} \label{lem:equiv-3}
The set of equations (\ref{eq:diff-2-sets}) is the same set as the
equations
\begin{align}
\sum_{a \in C_p^-,\ a \not\leq T^-} w(a) + \sum_{b \in C_p^+,\ b \not\geq T^+} w(b)
& = 2 w(p)
\textup{\qquad for all $p \in P$, $S_p = T^- \sqcup T^+$,
with $T^-$ an ideal and $T^+$ a filter}.
\label{eq:diff-2-sets-part}
\end{align}
\end{lemma}
\begin{proof} \disconnect
A particular equation in the set (\ref{eq:diff-2-sets}) is indexed
first by an $x \in \Idealf(P)$ and then a $p \in I_x$.
$p$ determines $C_p^+$ and $C_p^-$, and $p$ and $x$ together determine
$S^-_{px}$ and $S^+_{px}$.  Necessarily,
$S_p = S^+_{px} \sqcup S^+_{px}$,
so setting $T^- = S^-_{px}$ and $T^+ = S^+_{px}$ shows that the
equation is in the set (\ref{eq:diff-2-sets-part}).

Conversely, a particular equation in the set
(\ref{eq:diff-2-sets-part}) is indexed first by a $p \in P$ and then a
partition of $S_p$ into an ideal $T^-$ and a filter $T^+$.
Define $x = \{z \in P \mvert z < p\} \sqcup T^- = (p) \sqcup T^-$,
the principal ideal
of $P$ generated by $p$, less $p$ itself, plus $T^-$.
By construction, $x$ is an ideal of $P$, $p$ is in insertion point of $x$,
$S^-_{px} = T^-$ and $S^+_{px} = T^+$.
This shows that the equation is in the set (\ref{eq:diff-2-sets}).
\end{proof}

\paragraph{Orphans}
\begin{definition} \label{def:orphan}
Define $x \in C_p^+$ as an \emph{upward orphan} (relative to $p$)
if $x$ does not cover
any element of $S_p$, which is equivalent to $x$ covering only one
element of $N_p$ (namely $p$).
Similarly, we define $x \in C_p^-$ as a \emph{downward orphan}
(relative to $p$) if $x$ is not covered by
any element of $S_p$, which is equivalent to $x$ being covered by only one
element of $N_p$ (namely $p$).
\end{definition}

In (\ref{eq:diff-2-sets-part}), if we choose $T^- = \zeroslash$ and
$T^+ = S_p$, it becomes
\begin{align}
\sum_{a \in C_p^-} w(a) + \sum_{b \in C_p^+,\ b \not\geq S_p} w(b)
& = 2 w(p) \notag
\textup{\qquad for all $p \in P$}
\end{align}
By lem.~\ref{lem:x-plus}, $S_p$ contains every element covered by any element of
$C_p^+$ except for $p$, so the second sum is over all upward orphans.
Thus
\begin{lemma}
\begin{align}
\sum_{a \mvert a \lessdot p} w(a) +
\sum_{b \mvert b \gtrdot p,\ b \textup{ upward orphan}} w(b)
& = 2 w(p)
\textup{\qquad for all $p \in P$}
\label{eq:join-orphan}
\end{align}
\end{lemma}
Equation (\ref{eq:join-orphan}) allows the weight of a point to be
computed as half of the sum of the weights of the elements it covers,
if the point is not covered by any upward orphans.

If we choose $T^- = S_p$ and $T^+ = \zeroslash$, then
(\ref{eq:diff-2-sets-part}) becomes
\begin{lemma}
\begin{align}
\sum_{a \mvert a \lessdot p,\ a \textup{ downward orphan}} w(a) +
\sum_{b \mvert b \gtrdot p} w(b)
& = 2 w(p)
\textup{\qquad for all $p \in P$}
\label{eq:meet-orphan}
\end{align}
\end{lemma}

\paragraph{Partitioning \texorpdfstring{$S$}{S} into three sets}
Given $p \in P$, let us partition $S_p$ into three disjoint parts
$S_p = L \sqcup \{m\} \sqcup U$ where $L$ is an ideal of $S$, $U$ is
a filter of $S$, and thus $m$ is an element between $L$ and $U$.
In (\ref{eq:diff-2-sets-part}), setting its $T^- = L$ and its
$T^+ = \{m\}\sqcup U$ gives
\begin{align}
\sum_{a \mvert a \in C_p^-,\ a \not\leq L} w(a) +
\sum_{b \mvert b \in C_p^+,\ b \not\geq \{m\}\sqcup U} w(b)
& = 2 w(p) \label{eq:three-1} \\
\interject{and then setting its $T^- = L \sqcup \{m\}$ and
its $T^+ = U$ in (\ref{eq:diff-2-sets-part}) gives}
\sum_{a \mvert a \in C_p^-,\ a \not\leq L \sqcup \{m\}} w(a) +
\sum_{b \mvert b \in C_p^+,\ b \not\geq U} w(b)
& = 2 w(p) \label{eq:three-2}
\end{align}
Since $x \not\leq L\sqcup\{m\}$ implies $x \not\leq L$
and $x \not\geq \{m\}\sqcup U$ implies $x \not\geq U$,
subtracting (\ref{eq:three-2}) from (\ref{eq:three-1}) gives
\begin{align}
\sum_{a \mvert a \in C_p^-,\ a \not\leq L,\ \textup{not } a \not\leq L\sqcup\{m\}} w(a)
& = \sum_{b \mvert b \in C_p^+,\ b \not\geq U,\ \textup{not } b \not\geq \{m\}\sqcup U} w(b)
\notag
\end{align}
This is equivalent to
\begin{align}
\sum_{a \mvert a \in C_p^-,\ a \not\leq L,\ a \leq m} w(a)
& = \sum_{b \mvert b \in C_p^+,\ b \not\geq U,\ b \geq m} w(b)
\label{eq:L-m-U}
\end{align}

\begin{lemma} \label{lem:equiv-4}
Given a $p \in P$,
equation (\ref{eq:diff-2-sets-part}) holds for every partition of
$S_p$ into an ideal $T^-$ and a filter $T^+$ iff:
\begin{enumerate}
\item equation (\ref{eq:join-orphan}) holds for $p$ and
\item equation (\ref{eq:L-m-U}) holds for $p$ and
for every partition of $S_p$ into three
disjoint sets $S_p = L \sqcup \{m\} \sqcup U$, where L is an ideal of
$S_p$ and $U$ is a filter of $S_p$.
\end{enumerate}
\end{lemma}
\begin{proof} \disconnect
Regarding $\Rightarrow$: This has been proved by the above derivations of
(\ref{eq:join-orphan}), (\ref{eq:meet-orphan}), and (\ref{eq:L-m-U}).

Regarding $\Leftarrow$:
For a given $p$, we prove (\ref{eq:diff-2-sets-part}) for any partition
$S_p = T^- \sqcup T^+$ where $T^-$ is an ideal and $T^+$ is a filter
by induction upward on the size of $T^-$.
If $T^- = \zeroslash$, then by hypothesis (\ref{eq:join-orphan}) holds
for $p$,
which is equivalent to (\ref{eq:diff-2-sets-part}) for $T^- = \zeroslash$.

If $T^- \neq \zeroslash$, then choose $m$ as a maximal element of
$T^-$ and define $L = T^- \setminus \{m\}$ and $U = T^+$, so that
$S_p = L \sqcup \{m\} \sqcup U$ with $L$ being an ideal and $U$ being
a filter.
Since $|L| = |T^-|-1$,
by induction (\ref{eq:diff-2-sets-part}) holds for the partition
$S_p = L \sqcup (\{m\} \sqcup U)$, which is
\begin{align}
\sum_{a \in C_p^-,\ a \not\leq L} w(a) + \sum_{b \in C_p^+,\ b \not\geq \{m\} \sqcup U} w(b)
& = 2 w(p)
\label{eq:equiv-4-1}
\end{align}
By hypothesis, (\ref{eq:L-m-U}) holds for the partition
$S = L \sqcup \{m\} \sqcup U$, which is
\begin{align}
\sum_{a \mvert a \in C_p^-,\ a \not\leq L,\ a \leq m} w(a)
& = \sum_{b \mvert b \in C_p^+,\ b \not\geq U,\ b \geq m} w(b)
\label{eq:equiv-4-2}
\end{align}
Subtracting the left side of (\ref{eq:equiv-4-2}) from the first term of
(\ref{eq:equiv-4-1}) and adding the right side of (\ref{eq:equiv-4-2})
to the second term of (\ref{eq:equiv-4-1}) gives
\begin{align}
  \sum_{a \in C_p^-,\ a \not\leq L,\ a \not\leq m} w(a) +
  \sum_{b \in C_p^+,\ b \not\geq U} w(b)
& = 2 w(p)
\notag
\end{align}
This equation is the same as (\ref{eq:diff-2-sets-part})
for the partition $S_p = (L \sqcup \{m\}) \sqcup U = T^- \sqcup T^+$.

Thus by induction (\ref{eq:diff-2-sets-part}) holds for $p$ and
every partition of $S_p$ into an ideal $T^-$ and a filter $T^+$.
\end{proof}

We can replace (\ref{eq:join-orphan}) with (\ref{eq:meet-orphan}):

\begin{lemma} \label{lem:equiv-4-dual}
Given a $p \in P$,
equation (\ref{eq:diff-2-sets-part}) holds for every partition of
$S_p$ into an ideal $T^-$ and a filter $T^+$ iff:
\begin{enumerate}
\item equation (\ref{eq:meet-orphan}) holds for $p$ and
\item equation (\ref{eq:L-m-U}) holds for $p$ and
for every partition of $S_p$ into three
disjoint sets $S_p = L \sqcup \{m\} \sqcup U$, where L is an ideal of
$S_p$ and $U$ is a filter of $S_p$.
\end{enumerate}
\end{lemma}
\begin{proof} \disconnect
This is proved dually to the proof of lem.~\ref{lem:equiv-4}, using
(\ref{eq:meet-orphan}) in place of (\ref{eq:join-orphan}).
\end{proof}

\paragraph{Positive weighting} \label{para:positive}

\begin{assumption} \label{ass:positive}
Henceforth, we assume that $w$ is an
positive weighting, that is, the set of values $V$, the target of
$w$, is $\mbbZ$, and all of
the values of $w$ are positive.
\end{assumption}

This assumption is not a loss of generality for
our larger purpose, because if a non-positive weighting could be used to
construct a positive weighting using th.~\ref{th:truncate-fomin}
et seq.,
our analysis would reveal the resulting positive weighting directly.

\begin{lemma} \label{lem:isolated-simple}
For any $p \in P$ and for any $x,y \in S_p$,
if $x^+_p = y^+_p$ and $x^-_p = y^-_p$, then $x=y$.
\end{lemma}
\begin{proof} \disconnect
Assume there is $x,y \in S_p$, with $x \neq y$, $x^+_p = y^+_p$, and
$x^-_p = y^-_p$.
Since $x \neq y$, either $x \not\leq y$ or $y \not\leq x$.  Without loss of
generality, assume $y \not\leq x$.
Define
$X = \{s \in S_p \mvert s^-_p = x^-_p \textup{ and } s^+_p = x^+_p
\textup{ and } s \leq x \} = [x] \cap [x^-_p, x^+_p]$, the principal ideal
generated by $x$ within the interval $[x^-_p, x^+_p]$.
By construction, $x \in X$ and $y \not\in X$.

Given $a \in C_p^-$,
$a \leq X$ iff $a \leq x$, so $a \not\leq X$ iff $a \not\leq x$.
Given $b \in C_p^+$, if $b \geq$ some $z \in X$, $z^+_p = b$.
But since $z \leq x$, lem.~\ref{lem:x-plus} shows
$z^+_p = x^+_p$, so $b = z^+_p = x^+_p = y^+_p$ and $b \geq y \not\in X$.
Thus $b \geq S_p$ iff $b \geq S_p \setminus X$ and
$b \not\geq S_p$ iff $b \not\geq S_p \setminus X$.

Applying (\ref{eq:diff-2-sets-part}) to $S_p = \zeroslash \sqcup S_p$ gives
\begin{align}
\sum_{a \in C_p^-,\ a \not\leq \zeroslash} w(a) + \sum_{b \in C_p^+,\ b \not\geq S_p} w(b)
& = 2 w(p) \notag \\
\interject{which is equivalent to}
\sum_{a \in C_p^-} w(a) + \sum_{b \in C_p^+,\ b \not\geq S_p \setminus X} w(b)
& = 2 w(p) \label{eq:isolated-simple-1}
\end{align}
Applying (\ref{eq:diff-2-sets-part}) to
$S_p = X \sqcup (S_p \setminus X)$ gives
\begin{align}
  \sum_{a \in C_p^-,\ a \not\leq X} w(a) +
  \sum_{b \in C_p^+,\ b \not\geq S_p \setminus X} w(b)
& = 2 w(p) \notag \\
\interject{which is equivalent to}
  \sum_{a \in C_p^-,\ a \not\leq x} w(a) +
  \sum_{b \in C_p^+,\ b \not\geq S_p \setminus X} w(b)
& = 2 w(p) \label{eq:isolated-simple-2}
\end{align}
Subtracting (\ref{eq:isolated-simple-2}) from  (\ref{eq:isolated-simple-1})
gives $w(x^-_p) = 0$.
By the assumption that $w$ is positive, this is a contradiction.
Thus $x = y$.
\end{proof}

\begin{lemma} \label{lem:antichain}
For any $p \in P$, $S_p$ is an antichain.
\end{lemma}
\begin{proof} \disconnect
Assume there exists $x < y$ in $S_p$.
By lem.~\ref{lem:x-plus}, $x^+_p = y^+_p$ and $x^-_p = y^-_p$.
By lem.~\ref{lem:isolated-simple}, $x = y$, which is a contradiction.
Thus $S_p$ is an antichain.
\end{proof}

Thus, for any $p$, any subset of $S_p$ is both an ideal and a filter,
which allows (\ref{eq:diff-2-sets-part}) and (\ref{eq:L-m-U})
to be applied quite broadly.
Consider how unique-cover-modularity simplifies (\ref{eq:L-m-U}):
For any given $m \in S_p$, any $a$ on the left-hand side of
(\ref{eq:L-m-U}) can only be $m^-_p$.  Similarly,
any $b$ on the right-hand side can only be $m^+_p$.  Thus
(\ref{eq:L-m-U}) is equivalent to:
\begin{align}
\begin{rcases}
  \textup{if } m^-_p \not\leq L \textup{ } & w(m^-_p) \\
  \textup{otherwise }                    & 0
\end{rcases}
& =
\begin{cases}
  w(m^+_p) & \textup{ if } m^+_p \not\geq U \\
  0      & \textup{ otherwise}
\end{cases}
\label{eq:L-m-U-iso}
\end{align}

If we set $L=\zeroslash$ and $U=S_p\setminus \{m\}$,
\begin{align}
w(m^-_p) & =
\begin{cases}
  w(m^+_p) & \textup{ if } m^+_p \not\geq U \\
  0      & \textup{ otherwise}
\end{cases}
\label{eq:L-m-U-iso-1}
\end{align}
Given the assumption of unique-cover-modularity,
$m^+_p \geq z$ for some $z \in U$ iff
$m^+_p = z^+_p$ for some $z \in U$.
So $m^+_p \not\geq U$ iff there is no $z \in U$ with $m^+_p = z^+_p$.
Given the assumption of positivity, $w(m^-_p) > 0$, which is
contradicted by (\ref{eq:L-m-U-iso-1}) unless there is no $z \in S_p$
with $z^+_p = m^+_p$ and $z \neq m$.

Dually,
if we set $L=S_p\setminus \{m\}$ and $U=\zeroslash$,
we get a contradiction
unless there is no $z \in S_p$ with $z^-_p = m^-_p$ and $z \neq m$.

Thus for every $m \neq n$ in $S_p$, $m^-_p \neq n^-_p$ and $m^+_p \neq n^+_p$,
and so the set of pairs $\{(m^-_p, m^+_p) \mvert m \in S_p\}$
is a bijection between subsets of $C_p^-$ and $C_p^+$.
In addition, (\ref{eq:L-m-U-iso}) reduces to
\begin{align}
w(m^-_p) & = w(m^+_p) \textup{\qquad for all $p \in P$ and $m \in S_p$}
\label{eq:w-plus-minus}
\end{align}

\begin{lemma} \label{lem:equiv-5}
Given $p \in P$,
(\ref{eq:L-m-U}) holds for every partition of $S_p$ into three
disjoint sets $S_p = L \sqcup \{m\} \sqcup U$
with $L$ an ideal of $S_P$ and $U$ a filter of $S_p$
iff the following conditions are true:
\begin{enumerate}
\item the set of pairs $\{(m^-_p, m^+_p) \mvert m \in S_p\}$
is a bijection between subsets of $C_p^-$ and $C_p^+$ and
\item for every $m \in S_p$, $w(m^-_p) = w(m^+_p)$.
\end{enumerate}
\end{lemma}
\begin{proof} \disconnect
Regarding $\Rightarrow$: This has been shown by the above derivations.

Regarding $\Leftarrow$:
Assume $S_p = L \sqcup \{m\} \sqcup U$.
By the above derivation,
(\ref{eq:L-m-U}) is the equivalent to (\ref{eq:L-m-U-iso}).
By hypothesis (1) and lem.~\ref{lem:x-plus},
$m^-_p \leq z$ for some $z \in L$ iff
$m^-_p = z^-_p$ iff $m = z$.  But by construction, $m \not\in L$,
so $m^-_p \not\leq L$ is always true.  Dually, $m^+_p \not\leq U$
is always true.  Thus (\ref{eq:L-m-U-iso}) is equivalent to
$w(m^-_p) = w(m^+_p)$ which is assumed as hypothesis (2).
\end{proof}

At this point, let us summarize our progress:

\begin{lemma} \label{lem:equiv-6}
Given
\begin{enumerate}
\item[] (assump.~\ref{ass:distributive}) $L$ satisfies the criteria of
  a Fomin lattice other than \ref{i:fomin-A}\ref{i:fomin-A3},
  the weighted-differential property,
\item[] (assump.~\ref{ass:distributive}) $L$ has a $\zerohat$,
\item[] (assump.~\ref{ass:distributive}) $L$ is distributive,
\item[] (assump.~\ref{ass:unique-cover-modular}) $P$ is unique-cover-modular,
and
\item[] (assump.~\ref{ass:positive}) $w(\bullet)$ from $P$ into
  $V = \mbbZ$ has only positive values,
\end{enumerate}
then
$L$ satisfies (9) (and thus is a Fomin lattice) iff
\begin{enumerate}
\item[] (eq.~\ref{eq:minimal}) $r = \sum_{z \mvert z \textup{ minimal in } P} w(z)$,
\item[] (eq.~\ref{eq:join-orphan}) $\sum_{a \mvert a \lessdot p} w(a) +
\sum_{b \mvert b \gtrdot p,\ b \textup{ upward orphan}} w(b) = 2 w(p)$
for all $p \in P$,
\item[] (lem.~\ref{lem:equiv-5}(1)) the set of pairs
$\{(m^-_p, m^+_p) \mvert m \in S_p\}$ is a bijection between
subsets of $C_p^-$ and $C_p^+$ for all $p \in P$, and
\item[] (lem.~\ref{lem:equiv-5}(2)) $w(m^-_p) = w(m^+_p)$
for all $p \in P$, $m \in S_p$.
\end{enumerate}
\end{lemma}
\begin{proof} \disconnect
  This is proved by chaining together
  th.~\ref{th:differential-distributive}, lem.~\ref{lem:equiv-1},
  lem.~\ref{lem:equiv-2}, lem.~\ref{lem:equiv-3},
  lem.~\ref{lem:equiv-4}, and lem.~\ref{lem:equiv-5}.
\end{proof}

Similarly to lem.~\ref{lem:equiv-4} and~\ref{lem:equiv-4-dual}, we can
replace (\ref{eq:join-orphan}) with (\ref{eq:meet-orphan}):

\begin{lemma} \label{lem:equiv-6-dual}
Given
\begin{enumerate}
\item[] (assump.~\ref{ass:distributive}) $L$ satisfies the criteria of
  a Fomin lattice other than \ref{i:fomin-A}\ref{i:fomin-A3},
  the weighted-differential property,
\item[] (assump.~\ref{ass:distributive}) $L$ has a $\zerohat$,
\item[] (assump.~\ref{ass:distributive}) $L$ is distributive,
\item[] (assump.~\ref{ass:unique-cover-modular}) $P$ is unique-cover-modular,
and
\item[] (assump.~\ref{ass:positive}) $w(\bullet)$ from $P$ into
  $V = \mbbZ$ has only positive values,
\end{enumerate}
then
$L$ satisfies (9) (and thus is a Fomin lattice) iff
\begin{enumerate}
\item[] (eq.~\ref{eq:minimal}) $r = \sum_{z \mvert z \textup{ minimal in } P} w(z)$,
\item[] (eq.~\ref{eq:meet-orphan})
  $\sum_{a \mvert a \lessdot p,\ a \textup{ downward orphan}} w(a) +
\sum_{b \mvert b \gtrdot p} w(b) = 2 w(p)$
for all $p \in P$,
\item[] (lem.~\ref{lem:equiv-5}(1)) the set of pairs
$\{(m^-_p, m^+_p) \mvert m \in S_p\}$ is a bijection between
subsets of $C_p^-$ and $C_p^+$ for all $p \in P$, and
\item[] (lem.~\ref{lem:equiv-5}(2)) $w(m^-_p) = w(m^+_p)$
for all $p \in P$, $m \in S_p$.
\end{enumerate}
\end{lemma}
\begin{proof} \disconnect
  This is proved by chaining together
  th.~\ref{th:differential-distributive}, lem.~\ref{lem:equiv-1},
  lem.~\ref{lem:equiv-2}, lem.~\ref{lem:equiv-3},
  lem.~\ref{lem:equiv-4-dual}, and lem.~\ref{lem:equiv-5}.
\end{proof}

\paragraph{No upward or downward triple covers}

\begin{lemma} \label{lem:no-triples}
  No element of $P$ is covered by three or more distinct elements.
\end{lemma}
\begin{proof} \disconnect
The construction for this poof is shown in fig.~\ref{fig:Triple}.

\begin{figure}[htp]
\centering
\includegraphics[page=\ipeFigTriple,scale=0.9]{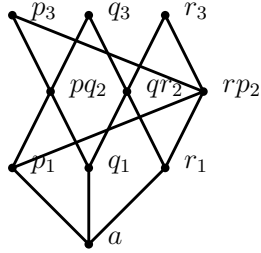}
\caption{Construction for the proof of lem.~\ref{lem:no-triples}}
\label{fig:Triple}
\end{figure}

Assume that we have $a \in P$ covered by distinct $p_1$, $q_1$ and
$r_1$.
By unique-cover-modularity, there is a $pq_2$ covering $p_1$ and $q_1$,
a $qr_2$ covering $q_1$ and $r_1$, and an $rp_2$ covering $r_1$ and
$p_1$.
By lem.~\ref{lem:isolated-simple} (applied to $p_1$, $q_1$, and $r_1$),
$pq_2$, $qr_2$, and $rp_2$ are distinct.
By (\ref{eq:w-plus-minus}) (applied to $p_1$, $q_1$, and $r_1$),
$w(a) = w(pq_2) = w(qr_2) = w(rp_2)$.

By unique-cover-modularity, there is a $p_3$ covering $rp_2$ and $pq_2$, a
$q_3$ covering $pq_2$ and $qr_2$, and an $r_3$ covering $qr_3$ and
$rp_3$.
Suppose two of $p_3$, $q_3$, and $r_3$ are the same, say $p_3$ is the
same as $q_3$.  Define $b = p_3 = q_3$, so $b$ covers $pq_2$,
$qr_2$, and $rp_2$.
Then by (\ref{eq:w-plus-minus}) (applied to $pq_2$, $qr_2$, and $rp_2$),
$w(b) = w(p_1) = w(q_1) = w(r_1)$.
Since $w$ is positive,
(\ref{eq:join-orphan}) implies $2w(pq_2) \geq w(p_1)+w(q_1) = 2w(b)$.
So $w(pq_2) \geq w(b)$ and similarly $w(qr_2) \geq w(b)$ and $w(rp_2) \geq w(b)$.
Applying (\ref{eq:join-orphan}) to $b$ and the above gives
$2 w(b) \geq w(pq_2) + w(qr_2) + w(rp_2) \geq 3w(b)$.
Since $w(b) > 0$, this is a contradiction.
Thus $p_3$, $q_3$, and $r_3$ are distinct.

By (\ref{eq:w-plus-minus}) (applied to $pq_2$, $qr_2$, and $rp_2$),
$w(p_3)=w(p_1)$, $w(q_3)=w(q_1)$, and $w(r_3)=w(r_1)$.
Applying (\ref{eq:join-orphan}) and positivity,
\begin{align}
2w(p_1)+2w(q_1)+2w(r_1) & = 2w(p_3) + 2w(q_3)+2w(r_3) \notag \\
2w(p_3) + 2w(q_3)+2w(r_3) & \geq 2w(pq_2) + 2w(qr_2)+2w(rp_2)
\label{eq:no-triple-3-2} \\
2w(pq_2) + 2w(qr_2)+2w(rp_2) & \geq 2w(p_1)+2w(q_1)+2w(r_1)
\label{eq:no-triple-2-1}
\end{align}
Thus
\begin{align}
w(p_1)+w(q_1)+w(r_1) & = w(pq_2) + w(qr_2)+w(rp_2) = w(p_3) + w(q_3)+w(r_3)
\notag
\end{align}
and the inequalities (\ref{eq:no-triple-3-2}) and
(\ref{eq:no-triple-2-1}) are equalities.
Equation~(\ref{eq:join-orphan})
and the assumption that $w$ is positive requires
that $pq_2$, $qr_2$, $rp_2$, $p_3$, $q_3$, and $r_3$ cover no elements
other than the ones mentioned above (and shown in fig.~\ref{fig:Triple})
and have no upward orphans.
Applying (\ref{eq:join-orphan}) to $pq_2$, $qr_2$, and $rp_2$ thus gives
\begin{align}
w(a) = w(pq_2) & = w(p_1)/2 + w(q_1)/2 \\
w(a) = w(qr_2) & = w(q_1)/2 + w(r_1)/2 \\
w(a) = w(rp_2) & = w(r_1)/2 + w(p_1)/2.
\end{align}
Together these prove $w(p_1) = w(q_1) = w(r_1) = w(a)$ and
$w(p_3)=w(q_3)=w(r_3)=w(a)$.

By (\ref{eq:meet-orphan}) and that $w$ is positive,
$2w(a) \geq w(p_1)+w(q_1)+w(r_1) = 3w(a)$.
Since $w(a) > 0$, this is a contradiction.

Thus, there is no $a \in P$ that is covered by three distinct elements.
\end{proof}

\begin{lemma} \label{lem:no-triples-dual}
  No element of $P$ covers three or more distinct elements.
\end{lemma}
\begin{proof} \disconnect
This is proved dually to lem.~\ref{lem:no-triples}, by interchanging the
use of (\ref{eq:join-orphan}) and (\ref{eq:meet-orphan}) and
replacing ``upward orphan'' with ``downward orphan''.
\end{proof}

\begin{remark} \label{rem:no-triples}
\textup{
Lem.~\ref{lem:no-triples} and~\ref{lem:no-triples-dual} are the deeper
explanation of the long-known fact\cite{Fom1994a}*{sec.~2.2} that
$\Idealf(\mbbN^k)$ for $k > 2$
has no positive weighting as a differential lattice.
Similarly, they show why
$\mbbSS_t = \Idealf(\mbbCP_{2t}) = \Idealf(t\textit{-SkewStrip})$
is a Fomin lattice (para.~\ref{para:known} item~(\ref{i:previous:SS})) but
$\Idealf(\mbbCP_{kt})$ with $k > 2$ has no
positive weightings.  (Compare $\mbbCP_{32} = \scrL_{3,2}$ in
\cite{El2025a}*{Fig.~2} with $\mbbCP_{33} = \scrL_{3,3}$ in
\cite{El2025a}*{Fig.~3}.)
}
\end{remark}

The equations implied by lem.~\ref{lem:equiv-6} for any particular $p \in P$
are characterized by the structure of its neighborhood $N_p$ as a poset.

\begin{lemma} \label{lem:neighborhoods}
$N_p$ is characterized up to poset isomorphism fixing $p$ by the numbers
$|C_p^-|$, $|C_p^+|$ and $|S_p|$:
All elements of $C_p^-$ are covered by $p$, all elements of $C_p^+$ cover
$p$, and $S_p$ forms a matching between subsets of $C_p^-$ and $C_p^+$, with
each element of $S_p$ covering a distinct element of $C_p^-$ and being
covered by a distinct element of $C_p^+$.
In addition, $|C_p^-| \leq 2$, $|C_p^+| \leq 2$,
$|S_p| \leq |C_p^+|$, $|S_p| \leq |C_p^-|$,
and $|N_p| = |C_p^-| + |C_p^+| + |S_p| + 1$.
\end{lemma}
\begin{proof} \disconnect
This follows from lem.~\ref{lem:antichain}, \ref{lem:equiv-5},
\ref{lem:no-triples}, and~\ref{lem:no-triples-dual}.
\end{proof}

For each possible set of values of $|C_p^-|$, $|C_p^+|$, $|S_p|$, and
$|N_p|$, fig.~\ref{fig:neighborhoods} shows
the resulting structure of $N_p$, and the consequent constraints
per lem.~\ref{lem:equiv-6}.

\disconnect
\vspace{\baselineskip}
\begin{figuregroup} \label{fig:neighborhoods}

\begin{figure}[htp]
\begin{tabular}{|c|c|c|c|c|c|}
\hline
$|C_p^+|$ & $|C_p^-|$ & $|S_p|$ & $|N_p|$ & Neighborhood & Constraints \\
\hline
0       & 0      & 0    & 1  &
\includegraphics[page=\ipeFigNzerozerozero,scale=0.9]{k-row-figs.pdf} &
\vbox{\hbox{$2w(p)=0$}} \\
\hline
0       & 1      & 0    & 2  &
\includegraphics[page=\ipeFigNzeroonezero,scale=0.9]{k-row-figs.pdf} &
\vbox{\hbox{$2w(p)=w(x_1)$}} \\
\hline
0       & 2      & 0    & 3  &
\includegraphics[page=\ipeFigNzerotwozero,scale=0.9]{k-row-figs.pdf} &
\vbox{\hbox{$2w(p)=w(x_1)+w(x_2)$}} \\
\hline
1       & 0      & 0    & 2  &
\includegraphics[page=\ipeFigNonezerozero,scale=0.9]{k-row-figs.pdf} &
\vbox{\hbox{$2w(p)=w(y_1)$}} \\
\hline
1       & 1      & 0    & 3  &
\includegraphics[page=\ipeFigNoneonezero,scale=0.9]{k-row-figs.pdf} &
\vbox{\hbox{$2w(p)=w(x_1)+w(y_1)$}} \\
\hline
1       & 1      & 1    & 4  &
\includegraphics[page=\ipeFigNoneoneone,scale=0.9]{k-row-figs.pdf} &
\vbox{\hbox{$2w(p)=w(x_1)$}\hbox{$w(x_1)=w(y_1)$}} \\
\hline
1       & 2      & 0    & 4  &
\includegraphics[page=\ipeFigNonetwozero,scale=0.9]{k-row-figs.pdf} &
\vbox{\hbox{$2w(p)=w(x_1)+w(x_2)+w(y_1)$}} \\
\hline
1       & 2      & 1    & 5  &
\includegraphics[page=\ipeFigNonetwoone,scale=0.9]{k-row-figs.pdf} &
\vbox{\hbox{$2w(p)=w(x_1)+w(x_2)$}\hbox{$w(x_1)=w(y_1)$}} \\
\hline
\end{tabular}
\caption{The possible neighborhoods of $p$ and their constraints, part 1}
\label{fig:neighborhoods-1}
\end{figure}

\begin{figure}[htp]
\begin{tabular}{|c|c|c|c|c|c|}
\hline
$|C_p^+|$ & $|C_p^-|$ & $|S_p|$ & $|N_p|$ & Neighborhood & Constraints \\
\hline
2       & 0      & 0    & 3  &
\includegraphics[page=\ipeFigNtwozerozero,scale=0.9]{k-row-figs.pdf} &
\vbox{\hbox{$2w(p)=w(y_1)+w(y_2)$}} \\
\hline
2       & 1      & 0    & 4  &
\includegraphics[page=\ipeFigNtwoonezero,scale=0.9]{k-row-figs.pdf} &
\vbox{\hbox{$2w(p)=w(x_1)+w(y_1)+w(y_2)$}} \\
\hline
2       & 1      & 1    & 5  &
\includegraphics[page=\ipeFigNtwooneone,scale=0.9]{k-row-figs.pdf} &
\vbox{\hbox{$2w(p)=w(x_1)+w(y_2)$}\hbox{$w(x_1)=w(y_1)$}} \\
\hline
2       & 2      & 0    & 5  &
\includegraphics[page=\ipeFigNtwotwozero,scale=0.9]{k-row-figs.pdf} &
\vbox{\hbox{$2w(p)=w(x_1)+w(x_2)+w(y_1)+w(y_2)$}} \\
\hline
2       & 2      & 1    & 6  &
\includegraphics[page=\ipeFigNtwotwoone,scale=0.9]{k-row-figs.pdf} &
\vbox{\hbox{$2w(p)=w(x_1)+w(x_2)+w(y_2)$}\hbox{$w(x_1)=w(y_1)$}} \\
\hline
2       & 2      & 2    & 7  &
\includegraphics[page=\ipeFigNtwotwotwo,scale=0.9]{k-row-figs.pdf} &
\vbox{\hbox{$2w(p)=w(x_1)+w(x_2)$}\hbox{$w(x_1)=w(y_1)$}\hbox{$w(x_2)=w(y_2)$}} \\
\hline
\end{tabular}
\caption{The possible neighborhoods of $p$ and their constraints, part 2}
\label{fig:neighborhoods-2}
\end{figure}

\end{figuregroup}

\section{Global characterization} \label{sec:global}

In this section, we continue by using the local characterization from
sec.~\ref{sec:local} to prove ``global'' facts about the poset $P$ of
points.  These culminate in an enumeration of a small number of
possible structures for $P$.

In this section, we continue the assumptions we have made in
section~\ref{sec:local}.
Thus, all points $p \in P$ have a neighborhood $N_p$ as
described by lem.~\ref{lem:neighborhoods}, and $w$ must
satisfy the consequent constraints per lem.~\ref{lem:equiv-6}.

\paragraph{Cartesian products}
If a lattice is
the cartesian product product of Fomin lattices,
para.~\ref{para:known} item~(\ref{i:previous:cartesian}) shows that it is
is a Fomin lattice.
It is straightforward that the product lattice preserves any of the
properties that we have been discussing (\ref{ass:distributive},
\ref{ass:unique-cover-modular}, and~\ref{ass:positive}) that are true of the
factors.
Conversely, if a Fomin lattice is the cartesian product of nontrivial
lattices, th.~\ref{th:cartesian}, lem.~\ref{lem:cartesian-ass},
\ref{lem:cartesian-union}, and~\ref{lem:ass-poset} show that its
factors are Fomin lattices that preserve any of the properties that are true
of the product lattice.  Thus, without excluding any
Fomin lattices from our classification, we can restrict our attention to
lattices that cannot be cartesian-factored.

\begin{assumption} \label{ass:non-factor}
Henceforth, we assume that $L$ cannot be factored into two nontrivial
lattices, which is equivalent to assuming that
$P$ cannot be partitioned into two disjoint non-empty
posets $P = P_1 \sqcup P_2$.
\end{assumption}

\paragraph{\texorpdfstring{$P$}{P} has a minimum element}

\begin{lemma} \label{lem:down-up}
Given any two elements $a, b \in P$, we
can connect them by a sequence of elements
$$a = x_0, x_1, x_2, \ldots,x_i, \ldots, x_{n-1}, x_n = b$$
with
$$a = x_0 \gtrdot x_1 \gtrdot x_2 \gtrdot \cdots \gtrdot x_i
\lessdot \cdots \lessdot x_{n-1} \lessdot x_n = b,$$
for some $0 \leq i \leq n$.
(That is, the sequence first descends, then ascends.)
\end{lemma}
\begin{proof} \disconnect
Because $P$ is locally finite and cannot be separated into the
disjoint union of two non-empty posets,
$a$ and $b$ can be connected with a sequence of elements
$a = x_0, x_1, x_2, \ldots, x_n = b$ where each adjacent pair
$x_i, x_{i+1}$ has either $x_i \lessdot x_{i+1}$ or
$x_i \gtrdot x_{i+1}$.  Given such a sequence, if there is an $i$ such
that $0 < i < n$ and $x_{i-1} \lessdot x_i \gtrdot x_{i+1}$,
that is, $x_i$ is a local maximum,
by the assumption of unique-cover-modularity, there is a $y$ such that
$x_{i-1} \gtrdot y \lessdot x_{i+1}$.
Thus $a = x_0, x_1, x_2, \ldots, x_{i-1}, y, x_{i+1}, \ldots x_n = b$
is another such sequence.

We can iterate this transformation on the sequence.
Each transformation preserves the number of adjacent pairs
$x_i \gtrdot x_{i+1}$ and the number of adjacent pairs
$x_i \lessdot x_{i+1}$, while moving the location of the
$x_i \gtrdot x_{i+1}$ pairs earlier in the sequence.
Thus any
series of such transformations must terminate in a sequence for which
there is no $i < i < n$ for which $x_i$ which is a local maximum.
Thus, any terminal sequence must have some $0 \leq i \leq n$ for which
for all $0 \leq j < i$, $x_j \gtrdot x_{j+1}$ and
for all $i \leq j < n$, $x_j \lessdot x_{j+1}$.
\end{proof}

\begin{lemma} \label{lem:P-minimum}
$P$ has a minimum element, which we denote $\zerohat_P$.
\end{lemma}
\begin{proof} \disconnect
Since $P$ is non-empty and finitary by lem.~\ref{lem:P-finite}(2), $P$
must contain at least one minimal element $m$.
Assume there is an element $a \not\geq m$ in $P$.
Then by lem.~\ref{lem:down-up}, there is a sequence of elements
$m = x_0, x_1, x_2, \ldots,x_n = a$
where for some $0 \leq i \leq n$,
for any $0 \leq j < i$, $x_j \gtrdot x_{j+1}$ and
for any $i \leq j < n$, $x_j \lessdot x_{j+1}$.
But since $m$ is minimal, there is no $y$ with
$m \gtrdot y$, and so $i = 0$.
If $n > 0$, then $m = x_0 \lessdot x_1 \leq x_n = a$.
That implies $m < a$, which contradicts that $a \not\geq m$.
So $n = 0$ and $m = x_0 = a$, which also contradicts $a \not\geq m$.
Thus there is no $a \not\geq m$ in $P$ and $m$ is the minimum of $P$.
\end{proof}

\begin{lemma} \label{lem:graded}
$P$ is graded.
We define $\rho$ as the rank function on $P$.
$\rho(\zerohat_P) = 0$.
\end{lemma}
\begin{proof} \disconnect
This proof is a straightforward generalization of
th.~\ref{th:graded}\cite{Birk1967a}*{\S II.8 Th.~14}, given
that $P$ is locally finite, is
unique-cover-modular, and has a minimum element $\zerohat_P$.
\end{proof}

\begin{remark} \label{rem:graded}
In lem.~\ref{lem:graded}, the hypothesis that $\zerohat_P$ exists is
necessary.  A counterexample is the poset ``3-Plait'' in
\cite{Fom1994a}*{Fig.~7 and Exam.~2.2.10}.
Further counterexamples are the posets of points of the lattices of
cylindrical partitions.\cite{El2025a}
\end{remark}

\paragraph{The bottom elements of \texorpdfstring{$P$}{P}}

First we define the sequence, finite or infinite, of elements at the
bottom of $P$ that form a chain, with each element (above
$\zerohat_P$) being the unique cover of the one below.

\begin{definition}
We define $B_1 = \zerohat_P$.
Then iteratively for each $i \geq 0$, if $B_i$ has a unique upward
cover, we define $B_{i+1}$ to be that cover.
If there is a $B_\bullet$ that does not have a unique cover (and thus
by lem.~\ref{lem:no-triples} has either zero or two covers), we define
$N$ to be the index of that $B_\bullet$.
Otherwise, we define $N$ to be $\infty$.
($B$ for ``bottom''.)
\end{definition}

In the case $N = \infty$, the $B_\bullet$ form an upward semi-infinite
path, and they are the entirety of $P$.
\begin{figure}[htp]
\begin{center}
\includegraphics[page=\ipeFigInfinitePath]{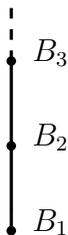}
\caption{The case $N = \infty$ where the $B_\bullet$ form an upward semi-infinite
path.}
\end{center}
\end{figure}
The constraints on $w$ are $w(B_1) = r$,
$2w(B_1) = w(B_2)$, and, for all $i \geq 2$,
$2w(B_i) = w(B_{i-1}) + w(B_{i+1})$,
which have the unique solution $w(B_i) = ir$.
Thus we have:
\begin{case} \label{case:path}
(para.~\ref{para:known} item~(\ref{i:previous:N}))
$P$ is the upward semi-infinite path $B_i$ for $i \geq 1$ with
$r > 0$ and $w(B_i) = ir$. $L$ is isomorphic to $\mbbN$.
\end{case}

If $N$ is not $\infty$, then there are two cases:  $B_N$ has zero
upward covers and $B_N$ has two upward covers.  If $B_N$ has zero
upward covers, then $P$ consists of the finite path
$B_0 \lessdot B_1 \lessdot B_2 \lessdot \cdots \lessdot B_N$.
\begin{figure}[htp]
\begin{center}
\includegraphics[page=\ipeFigFinitePath]{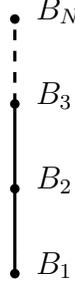}
\caption{The case $N$ is not $\infty$, $B_N$ has zero
upward covers, and the $B_\bullet$ form a finite path.}
\end{center}
\end{figure}
By lem.~\ref{lem:equiv-6}, the constraints on $w$ are
\begin{itemize}
\item $w(B_1) = r$;
\item if $N = 1$, $2w(B_1) = 0$;
\item if $N \geq 2$, $2w(B_1) = w(B_2)$;
\item for all $2 \leq i < N$, $2w(B_i) = w(B_{i-1}) + w(B_{i+1})$; and
\item if $N \geq 2$, $2w(B_N) = w(B_{N-1})$ .
\end{itemize}
For all finite $N$, these imply $r = 0$ and all $w(B_i) = 0$,
so $w$ is not positive.

\paragraph{\texorpdfstring{$P$}{P} is two-dimensional} \label{para:two-dim}

The remaining case is that $B_N$ has exactly two upward covers.

\begin{definition}
We define $P^+$ as the elements of $P$ that are not a $B_i$ for which
$i < N$.
\end{definition}

\begin{lemma}
$P^+$ has at least three elements: one is $B_N$
and two are the covers of $B_N$.
All $p \in P^+$ are $\geq B_N$, and
each has a finite grade $\rho(p) \geq \rho(B_N)$ by
lem.~\ref{lem:graded}.
\end{lemma}

We will ``draw the Hasse diagram of $P$ in the plane'' by defining a set of
locations $T_{xy}$ for all integers $x \geq 0$ and $y \geq 0$.
As a poset, $T \cong \mbbN \times \mbbN$.
Then we assign every $p \in P^+$ to a distinct
location $T_{xy}$ with $x+y = \rho(p) - \rho(B_N) = \rho(p) - N$.
($T$ for ``top elements''.)
That is, $B_N$ is assigned to $T_{00}$ and elements of $P^+$ are
assigned ``upward'' from $T_{00}$.
If $p \lessdot q$ in $P^+$ and $p$ is assigned to $T_{xy}$ then
$q$ is assigned to either $T_{x+1,y}$ or $T_{x,y+1}$.
Conversely, if both $T_{xy}$ and either $T_{x+1,y}$ or $T_{x,y+1}$ are
assigned, the latter covers the former.
Intuitively, the $T_{\bullet\bullet}$ form a quarter-plane extending
upward, with
$T_{00}$ at the bottom center, the $x$ coordinate increasing to the
upper-left and the $y$ coordinate increasing to the upper-right.
Note that all elements of $P^+$ are assigned to distinct locations
$T_{\bullet\bullet}$, but some locations may not have an assigned element.
We abuse notation by using $T_{xy}$ to denote both the location and,
if an element of $P^+$ is assigned to that location, that element of
$P^+$.

We define the coordinate $v = x+y$; $v$
is how far $T_{xy}$ is vertically above $T_{00}$.
We define the coordinate $h = y-x$; $h$
is how far $T_{xy}$ is horizontally left of $T_{00}$ (if negative) or
right of $T_{00}$ (if positive).
Note that while $v$ ranges over all nonnegative integers,
$h$ ranges over integers $-v \leq h \leq v$ which have the same parity
as $v$.
The assignment of elements of $P^+$ realizes its
Hasse diagram, with each cover relation being either one step to the
upper left (incrementing $x$) or one step to the upper right
(incrementing $y$).
An example assignment is shown in fig.~\ref{fig:U}.

\begin{figure}[htp]
\centering
\includegraphics[page=\ipeFigU,scale=1]{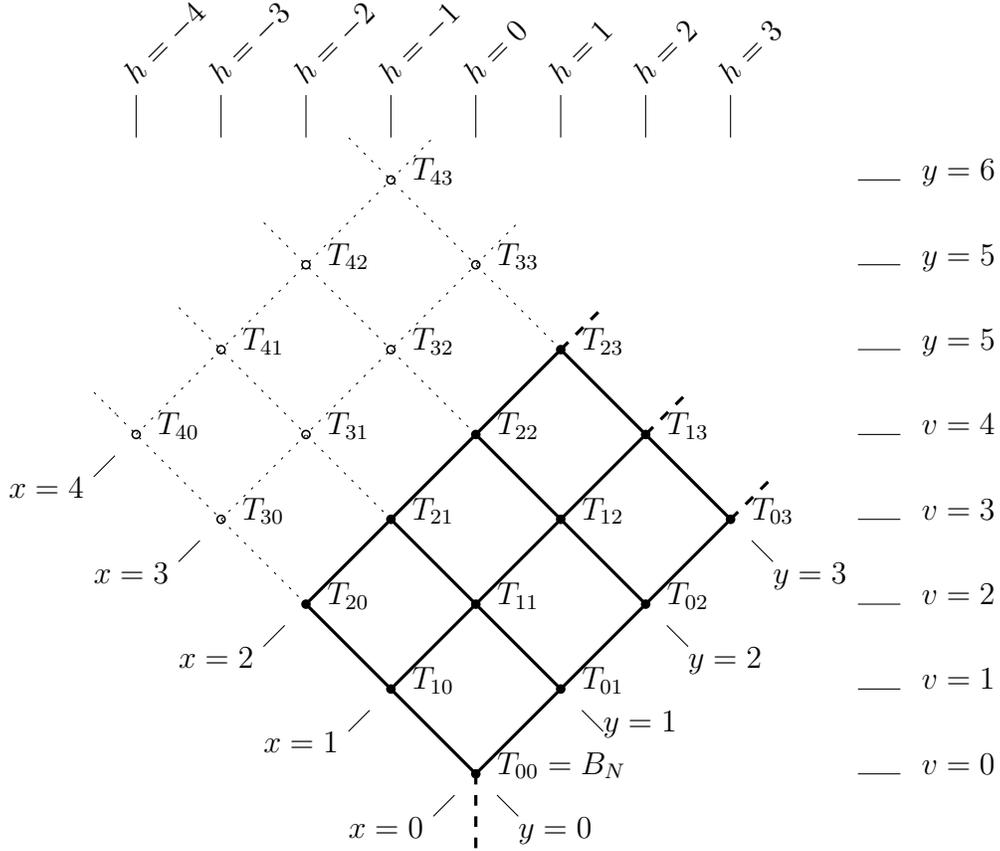}
\caption{An example of assigning elements of $P^+$ to
  locations $T_{\bullet\bullet}$.  Solid circles are assigned
  locations, empty circles are unassigned locations.  Solid lines are
  coverings in $P$.
  Further elements are assigned to the upper-right of $T_{03}$,
  $T_{13}$, and $T_{23}$.
  $B_i$ for $i < N$ are assigned downward from $T_{00}$.}
\label{fig:U}
\end{figure}

\begin{figure}[htp]
\centering
\includegraphics[page=\ipeFigAssignment,scale=1]{k-row-figs.pdf}
\caption{An example of assigning elements of $P^+$ with rank
  $v_\tnew+R$ to locations with $v=v_\tnew$ in steps (3a--f).
  $T_{ab}$ has $v=v_\tnew-1$ and $h=h^\tmin_{v_\tnew-1}$ and
  $T_{cd}$ has $v=v_\tnew-1$ and $h=h^\tmax_{v_\tnew-1}$.}
\label{fig:Assignment}
\end{figure}

We construct the assignments by first assigning $T_{00} = B_N$ and
successively assigning each grade of $P^+$ above $N$ to
locations so as to satisfy these rules.

\begin{lemma} \label{lem:assignment}
We can assign the elements of $P^+$ to locations in
$\{ T_{xy} \mvert x \geq 0 \textup{ and } y \geq 0 \}$ by the
following process, which has the following properties:
\begin{enumerate}
\item \label{i:assignment:1} Assign $B_N$ to $T_{00}$.
\item \label{i:assignment:2}
  Assign rank $N-1$, the two covers of $B_N$, arbitrarily to
  $T_{10}$ and $T_{01}$.
\item \label{i:assignment:3} Iteratively consider the grades of $P^+$
  from $N+2$ upward.
  Terminate if we reach a grade that is empty (all elements of
  $P^+$ have been assigned):
\begin{enumerate}
  \item \label{i:assignment:3a}
  Within  grade $R$, we will assign the elements of that grade to distinct
  locations $T_{xy}$ with $x+y=v_\tnew = R - N$.
  At this point, these iteration hypotheses are true of the $T_{xy}$
  with $x+y < v_\tnew$:
  \begin{enumerate}
    \item \label{i:assignment:3a-first}
    The elements assigned to locations with a particular value of
    $v < v_\tnew$ have
    values $h$ that are a contiguous subset of the integers
    $-v \leq h \leq v$ which have the same parity as $v$.
    Define $h^\tmin_v$ to be the minimum $h$ of an assigned
    location with a particular value of $v$ and
    define $h^\tmax_v$ to be the maximum $h$ of an assigned location
    with a particular value of $v$.
    Similarly, define $x^\tmin_v$, $x^\tmax_v$, $y^\tmin_v$, and
    $y^\tmax_v$ as the minimum and maximum $x$ and $y$ coordinates of
    assigned locations with this value of $v$.
    \item For any $0 < v < v_\tnew$,
    $h^\tmin_v = h^\tmin_{v-1} \pm 1$ and
    $h^\tmax_v = h^\tmax_{v-1} \pm 1$.
    Thus $x^\tmin_v = x^\tmin_{v-1}$ or $x^\tmin_{v-1}+1$,
    $x^\tmax_v = x^\tmax_{v-1}$ or $x^\tmax_{v-1}+1$,
    $y^\tmin_v = y^\tmin_{v-1}$ or $y^\tmin_{v-1}+1$, and
    $y^\tmax_v = y^\tmax_{v-1}$ or $y^\tmax_{v-1}+1$,
    \item If $p \lessdot q$ in $P^+$ are assigned at this point, then $q$ is
      assigned to some $T_{xy}$ and $p$ is assigned to either
      $T_{x-1,y}$ or $T_{x,y-1}$.
    \item \label{i:assignment:3a-last}
      Conversely, if $T_{xy}$ is assigned and either of
      $T_{x-1,y}$ or $T_{x,y-1}$ is assigned, then the former covers the
      latter in $P^+$.
  \end{enumerate}
    \item \label{i:assignment:3b}
    Consider any adjacent pair of assigned locations with
    $v = v_\tnew-1$, $T_{xy}$ and $T_{x-1,y+1}$, with $x+y=v_\tnew-1$,
    $h^\tmin_{v_\tnew-1} \leq y-x$, and $(y+1)-(x-1) \leq h^\tmax_{v_\tnew-1}$.
    $T_{xy}$ and $T_{x-1,y+1}$ cover $T_{xy} \wedge T_{x-1,y+1} = T_{x-1,y}$ and
    so they must be covered by $T_{xy} \vee T_{x-1,y+1}$, which has
    rank $R$ in $P^+$.
    Assign $T_{xy} \vee T_{x-1,y+1}$ to location $T_{x,y+1}$.
    By lem.~\ref{lem:no-triples-dual},  $T_{xy} \vee T_{x-1,y+1}$
    covers no element other than $T_{xy}$ and $T_{x-1,y+1}$.
    \item \label{i:assignment:3c}
    Every element of rank $R$ not yet assigned necessarily covers
    at least one element of rank $R-1$.  But any element of rank
    $R-1$ with $h^\tmin_{v_\tnew-1} < h < h^\tmax_{v_\tnew-1}$ already
    has two elements that cover it that were assigned in
    step (\ref{i:assignment:3b})
    that cover it, so the only elements of rank $R-1$ that it might cover are
    \begin{enumerate}
      \item the assigned location with minimum $h$,
      $T_{ab} = T_{x^\tmax_{v_{\tnew-1}},y^\tmin_{v_{\tnew-1}}}$ and/or
      \item the assigned location with maximum $h$,
      $T_{cd} = T_{x^\tmin_{v_{\tnew-1}},y^\tmax_{v_{\tnew-1}}}$
    \end{enumerate}
    (which might be the same location).
    However, it cannot cover both of these elements if they are
    distinct.  If it did, then both of these elements would cover an
    element of rank $R-2$, which would have caused this element (of rank
    $R$) to have been assigned a location in step
    (\ref{i:assignment:3b}).
  \item \label{i:assignment:3d} If an unassigned element covers $T_{ab}$, but
    $T_{ab}$ has another assigned cover (which must be $T_{a,b+1}$),
    assign this element to $T_{a+1,b}$.  There cannot be two such
    elements in a single rank, as then $T_{ab}$ would be covered by
    three elements.
  \item \label{i:assignment:3e} If an unassigned element covers $T_{pq}$, but
    $T_{pq}$ has another assigned cover (which must be $T_{p+1,q}$),
    assign this element to $T_{p,q+1}$.  There cannot be two such
    elements in a single rank, as then $T_{pq}$ would be covered by
    three elements.
  \item \label{i:assignment:3f}
    The remaining case is $T_{ab} = T_{pq}$ and unassigned
    elements remain.  By lem.~\ref{lem:no-triples}, there are at most
    two such elements, which can be assigned arbitrarily to
    $T_{a+1,b}$ and $T_{a,b+1}$.  (Ultimately, we will prove that
    this only happens in rank $N+1$, which was handled by step
    (\ref{i:assignment:2}).)
  \item At this point, all elements in rank $R$ of $P^+$ are assigned
    locations, and the iteration hypotheses
    (\ref{i:assignment:3a-first})--(\ref{i:assignment:3a-last})
    are true of the $T_{xy}$ with $x+y \leq v_\tnew$.
  \end{enumerate}
  \end{enumerate}
\end{lemma}

The algorithm is nondeterministic in cases (\ref{i:assignment:2}) and
(\ref{i:assignment:3f}); either
cover of the unique element in one rank may be assigned to either of
the two locations in the next rank.  In all other steps, it
is deterministic.  The typical operation of steps (3b--f) is shown in
fig.~\ref{fig:Assignment}.

\begin{proof} \disconnect
The induction properties stated in the lemma and the correctness of the
algorithm are proved by mutual induction.
The details are tedious but straightforward.
\end{proof}

Because the assigned locations of each grade are contiguous and very
similar to the assigned locations of the grades above and below, the
possible $P^+$ are determined up to isomorphism by the outline of the
set of assigned
locations, the values $h^\tmin_\bullet$ and $h^\tmax_\bullet$.

\paragraph{The bottom of \texorpdfstring{$P^+$}{P\textplussuperior}}
We can now examine the constraints implied by the neighborhoods of $P$.

\begin{definition} \label{def:m-delta}
We define $w_{xy} = w(T_{xy})$ when $T_{xy}$ is allocated,
$m = w(B_N) = w_{00}$,
$\delta^- = w_{10}-m$, and $\delta^+ = w_{01}-m$.
If $N > 1$, we define $\epsilon = w_{B_{N-1}} - m$ and otherwise
$\epsilon = -m$.
\end{definition}
($m$ for the ``weight of the \emph{middle} element of $P^+$'', namely $T_{00}=B_N$;
$\delta^-$ for ``the \emph{difference} in the
\emph{negative} direction of $h$'', that is, toward the left side of $P^+$;
$\delta^+$ for ``the \emph{difference} in the
\emph{positive} direction of $h$'', that is, toward the right side of $P^+$.)
Both $\delta^-$ and $\delta^+$ are well-defined since $T_{00} = B_N$
is assumed to have two upward covers, which are $T_{10}$ and $T_{01}$.

\begin{lemma} \label{lem:tail}
The constraints on $w$ due to applying lem.~\ref{lem:equiv-6} to $B_1$
and all the $B_i$ for $i < N$ are
\begin{equation}
w(B_i) = ir \textup{\qquad\qquad for all } 1 \leq i \leq N \label{eq:tail}
\end{equation}
In particular, $m = Nr$, $\epsilon = -r < 0$, and $N = -m/\epsilon$.
\end{lemma}

\newcommand{\xL}{{x_L}}
\newcommand{\yL}{{y_L}}

\begin{lemma}
The constraints on $w$ due to applying lem.~\ref{lem:equiv-6} to
$T_{\bullet 0}$ imply either
\begin{enumerate}
\item $\delta^- \geq 0$\upcolon
  we define $\xL =  \infty$,
  all $T_{\bullet 0}$ are allocated, and
  $w_{i0} = m + i\delta^-$ for all $0 \leq i$; or
\item $\delta^- < 0$\upcolon
  we define $\xL = -m/\delta^-$, $\xL$ is a positive integer $\geq 2$,
  $T_{i0}$ is allocated iff $i < \xL$, and
  $w_{i0} = m + i\delta^-$ for all $0 \leq i < \xL$.
\end{enumerate}
\end{lemma}
\begin{proof} \disconnect
This is proved by finding the smallest $I$ for which $T_{i0}$ is not
allocated, if one exists.  Because $T_{01}$ is allocated,
unique-modular-covering guarantees that for any $T_{i0}$ with $i \leq I$
(any $i$ if $I$ does not exist),
$T_{i1}$ is allocated and is not an upward orphan.
The constraints on $w_{i0}$ are that $2w_{i0} = w_{i-1,0} + w_{i+1,0}$
for any $0 < i < I$ and $2w_{i0} = w_{i-1,0}$ for $i = I$ (if $I$ exists).
Thus, $w_{i0} = m + i\delta^-$ for any $0 \leq i < I$ (and for any
$0 \leq i$ if $I$ does not exist).
If no such $I$ exists, conclusion case (1) follows, since positivity
for all of the $w_{\bullet 0}$ requires $\delta^- \geq 0$.
If $I$ exists, conclusion case (2) follows, since the constraint on
$w_{I-1,0}$ requires it to equal $-\delta^-$.
\end{proof}

\begin{lemma}
Similarly, the constraints on $w$ due to applying lem.~\ref{lem:equiv-6} to
$T_{0\bullet}$ imply either
\begin{enumerate}
\item $\delta^+ \geq 0$\upcolon
  we define $\yL =  \infty$,
  all $T_{0\bullet}$ are allocated, and
  $w_{0i} = m + i\delta^+$ for all $0 \leq i$; or
\item $\delta^+ < 0$\upcolon
  we define $\yL = -m/\delta^+$, $\yL$ is a positive integer $\geq 2$,
  $T_{0i}$ is allocated iff $i < \yL$, and
  $w_{0i} = m + i\delta^+$ for all $0 \leq i < \yL$.
\end{enumerate}
\end{lemma}

\begin{definition}
We define $\Delta^- = -m/\delta^-$ and $\Delta^+ = -m/\delta^+$,
defining $-m/0 = \infty$.
This parallels the relationship between $N$ and $\epsilon$, $N = -m/\epsilon$.
\end{definition}

\begin{lemma}
  $\Delta^-$ is either $\infty$ (if $\delta^- = 0$),
  a negative real number (if $\delta^- > 0$), or the
  reciprocal of a positive integer (which is $\xL$, if $\delta^- < 0$).
  Similarly for $\Delta^+$ w.r.t.\ $\yL$.
\end{lemma}

\newcommand{\recip}[1]{\frac{1}{#1}}

\begin{lemma}
The constraint on $T_{00}$ is equivalent to
\begin{equation}
\recip{\Delta^-} + \recip{\Delta^+} + \recip{N} = 1, \label{eq:delta-N}
\end{equation}
defining $1/\infty = 0$.
\end{lemma}
\begin{proof}
Applying lem.~\ref{lem:equiv-6} to $T_{00}$ yields
$2w_{00} = (w_{00}+\epsilon) + w_{10} + w_{01}$, since $T_{10}$ and
$T_{01}$ both exist and are upward orphans of $T_{00}$.  Dividing this
equation by $-m$ gives the conclusion.
\end{proof}

\begin{lemma} \label{lem:sum-of-recip}
Not considering the order of the terms, simple enumeration shows the
decompositions of $1$ as a sum of reciprocals as in (\ref{eq:delta-N}) are:
\begin{alignat*}{3}
  & \recip{1} & + \recip{n} & + (-\recip{n})
      \textup{\quad for any integer $n \geq 1$} \\
  & \recip{1} & + \recip{\infty} & + \recip{\infty} \\
  & \recip{2} & + \recip{2} & + \recip{\infty} \\
  & \recip{2} & + \recip{3} & + \recip{6} \\
  & \recip{2} & + \recip{4} & + \recip{4} \\
  & \recip{3} & + \recip{3} & + \recip{3}
\end{alignat*}
\end{lemma}

\paragraph{The possible cases for \texorpdfstring{$P$}{P}}

\begin{lemma} \label{lem:9-cases}
Since
\begin{enumerate}
  \item the possible decompositions of 1 as a sum of reciprocals are
    listed in lem.~\ref{lem:sum-of-recip},
  \item $N$ must be positive,
  \item given the left-right symmetry of $P$, we
    conventionally choose $\xL \geq \yL$, and
  \item $T_{10}$ and $T_{01}$ are allocated, so $\Delta^- \neq 1$
        and $\Delta^+ \neq 1$,
\end{enumerate}
the possible cases for $P$ are limited to:

\disconnect
\hspace{0.5in}
{
\renewcommand{\arraystretch}{1.4}
\textup{
\begin{tabular}{|l|l||l|l||l||l|}
\hline $\Delta^-$ & $\xL$ & $\Delta^+$ & $\yL$ & $N$ & Lattice \\
\hline $-n$ for $n \geq 2$ & $\infty$ & $n$ & $n$ & 1 & $\mbbY_n$ \\
\hline $\infty$ & $\infty$ & $\infty$ & $\infty$ & 1 & $\mbbY$ \\
\hline $\infty$ & $\infty$ & 2 & 2 & 2 & $\mbbSY$ \\
\hline 3 & 3 & 2 & 2 & 6 & none (lem.~\forwardref{lem:xy-32} \\
\hline 6 & 6 & 2 & 2 & 3 & none (lem.~\forwardref{lem:xy-62} \\
\hline 6 & 6 & 3 & 3 & 2 & none (lem.~\forwardref{lem:xy-ge-3}) \\
\hline 4 & 4 & 2 & 2 & 4 & none (lem.~\forwardref{lem:xy-42}) \\
\hline 4 & 4 & 4 & 4 & 2 & none (lem.~\forwardref{lem:xy-ge-3}) \\
\hline 3 & 3 & 3 & 3 & 3 & none (lem.~\forwardref{lem:xy-ge-3}) \\
\hline
\end{tabular}
}
}
\vspace{\baselineskip}
\end{lemma}

\begin{remark}
If we do not require $\zerohat$, then we remove the restriction $N>0$
in lem.~\ref{lem:9-cases}.
That allows the additional case $\Delta^- = \Delta^+ = 2$,
and $N = \infty$.  Developing that case leads to the Fomin lattices
para.~\ref{para:known}(\ref{i:previous:twos}).
\end{remark}

\begin{lemma} \label{lem:xy-ge-3}
  It is not possible that $\xL$ and $\yL$ are not $\infty$ and are $> 2$.
\end{lemma}
\begin{proof} \disconnect
If $\xL$ and $\yL$ are not $\infty$ and are $> 2$, then $T_{20}$ and
$T_{02}$ are allocated
and neither $T_{21}$ nor $T_{12}$ are upward orphans of $T_{11}$.  Thus
the constraints of lem.~\ref{lem:equiv-6} for $T_{11}$ simplifies to
$0 = \delta^- + \delta^+$.
But since $\xL$ and $\yL$ are finite, $\delta^-<0$ and $\delta^+ < 0$,
which is a contradiction.
\end{proof}

\begin{lemma} \label{lem:xy-32}
  It is not possible that $\xL = 3$ and $\yL = 2$.
\end{lemma}
\begin{proof} \disconnect
\begin{figure}[htp]
\centering
  \includegraphics[page=\ipeFigCthreetwo,scale=1]{k-row-figs.pdf}
\caption{The case $\xL = 3$ and $\yL = 2$:
Elements are labeled with their weights in multiples of $m$.
Coverings marked with a stroke are upward orphan coverings.
Double vertical strokes connect weights that are equal by
lem.~\ref{lem:equiv-5}(2).
Empty circles are $T_{\bullet\bullet}$ that are not allocated.
The constraint on $T_{44}$ must be violated.}
\label{fig:C32}
\end{figure}

Applying the constraints of lem.~\ref{lem:equiv-6} to
$T_{\bullet\bullet}$ in increasing order of
$v$ straightforwardly determines the allocations, non-allocations, and
weights shown in fig.~\ref{fig:C32}.  But of necessity,
$w_{44} = w_{33} = m$, which violates the constraint on $T_{44}$.
\end{proof}

\begin{lemma} \label{lem:xy-42}
  It is not possible that $\xL = 4$ and $\yL = 2$.
\end{lemma}
\begin{proof} \disconnect
\begin{figure}[htp]
\centering
  \includegraphics[page=\ipeFigCfourtwo,scale=1]{k-row-figs.pdf}
\caption{The case $\xL = 4$ and $\yL = 2$:
Elements are labeled with their weights in multiples of $m$.
Coverings marked with a stroke are upward orphan coverings.
Double vertical strokes connect weights that are equal by
lem.~\ref{lem:equiv-5}(2).
Empty circles are $T_{\bullet\bullet}$ that are not allocated.
The constraint on $T_{22}$ must be violated.}
\label{fig:C42}
\end{figure}

Applying the constraints of lem.~\ref{lem:equiv-6} to
$T_{\bullet\bullet}$ in increasing order of
$v$ straightforwardly determines the allocations, non-allocations, and
weights shown in fig.~\ref{fig:C42}.  But of necessity,
$w_{22} = w_{11} = m$, which violates the constraint on $T_{22}$.
\end{proof}

\begin{lemma} \label{lem:xy-62}
  It is not possible that $\xL = 6$ and $\yL = 2$.
\end{lemma}
\begin{proof} \disconnect
\begin{figure}[htp]
\centering
  \includegraphics[page=\ipeFigCsixtwo,scale=1]{k-row-figs.pdf}
\caption{The case $\xL = 6$ and $\yL = 2$:
Elements are labeled with their weights in multiples of $m$.
Coverings marked with a stroke are upward orphan coverings.
Double vertical strokes connect weights that are equal by
lem.~\ref{lem:equiv-5}(2).
Empty circles are $T_{\bullet\bullet}$ that are not allocated.
The constraint on $T_{22}$ must be violated.}
\label{fig:C62}
\end{figure}

Applying the constraints of lem.~\ref{lem:equiv-6} to
$T_{\bullet\bullet}$ in increasing order of
$v$ straightforwardly determines the allocations, non-allocations, and
weights shown in fig.~\ref{fig:C62}.  But of necessity,
$w_{22} = w_{11} = m$, which violates the constraint on $T_{22}$.
\end{proof}

\begin{lemma} \label{lem:3-real-cases}
For the first three cases of lem.~\ref{lem:9-cases},
$L$ must be:
\begin{enumerate}
  \item for $\Delta^- = -n$ and $\Delta^+ = n$, $L$ is isomorphic to $\mbbY_n$ with
    weights that are $m/n$ times the canonical weights
    (para.~\ref{para:known}(\ref{i:previous:Yk}) and fig.~\ref{fig:Y3}),
  \item for $\Delta^- = \Delta^+ = \infty$, $L$ is isomorphic to $\mbbY$ with
    weights that are $m$ times the canonical weights
    (para.~\ref{para:known}(\ref{i:previous:Y}) and fig.~\ref{fig:Y}),
  \item for $\Delta^- = \infty$ and $\Delta^+ = 2$, $L$ is isomorphic
    to $\mbbSY$ with
    weights that are $m/2$ times the canonical weights
    (para.~\ref{para:known}(\ref{i:previous:SY}) and fig.~\ref{fig:SY}),
\end{enumerate}
with $P$ being the poset of join-irreducibles of $L$ with the
corresponding weights.
\end{lemma}
\begin{proof} \disconnect
In all of these cases, the values $w_{B_i}$, $w_{x0}$, and
$w_{0y}$ are determined (relative to $m$) by $\xL$, $\yL$, and
$N$.  Which other elements of $T_{\bullet\bullet}$ are allocated and
their weights are straightforwardly determined by the constraints of
lem.~\ref{lem:equiv-6} on
them (working in order of increasing $v$), and necessarily result in
the known lattices listed in the table.
\end{proof}

\section{Classification theorem} \label{sec:classification}

\begin{theorem} \label{th:classification}
The only Fomin lattices $L$ that satisfy the assumptions
\begin{enumerate}
\item[] (assump.~\ref{ass:distributive}) The lattice $L$ is distributive.
  We define the poset
  of points (join-irreducibles) to be $P$, so $L = \Idealf(P)$.
\item[] (assump.~\ref{ass:unique-cover-modular}) $P$ is unique-cover-modular.
\item[] (assump.~\ref{ass:positive})
The weighting $w$ is positive, that is, the set of values $V$, the target of
$w$, is $\mbbZ$, and all of
the values of $w$ are positive.  The differential degree $r$ is positive.
\item[] (assump.~\ref{ass:non-factor})
The lattice cannot be factored into two nontrivial
lattices, which is equivalent to assuming that
$P$ cannot be partitioned into two disjoint non-empty
posets $P = P_1 \sqcup P_2$.
\end{enumerate}
are the cases
\begin{enumerate}
\item[$\mbbY$\upcolon] (para.~\ref{para:known}(\ref{i:previous:Y}))
$P$ is isomorphic to the quadrant of the plane,
$\{ (x,y) \mvert \textup{$x$ and $y$ are integers $\geq 0$}\}$,
and $L$ is isomorphic to $\mbbY$, Young's lattice, the lattice of partitions,
with weighting $r$ times the canonical weighting of 1.
\item[$\mbbSY$\upcolon] (para.~\ref{para:known}(\ref{i:previous:SY}))
$P$ is isomorphic to the octant of the plane,
$\{ (x,y) \mvert \textup{$x$ and $y$ are integers with
$x \geq 0$ and $0 \leq y \leq x$}\}$,
and $L$ is isomorphic to $\mbbSY$, the shifted Young's lattice, the lattice of
partitions into distinct parts,
with weighting $r$ times the canonical weighting of 1 and 2.
\item[$\mbbY_\bullet$\upcolon] (para.~\ref{para:known}(\ref{i:previous:Yk}))
$P$ is isomorphic to the $k$-row strip for some $k \geq 1$,
$\{ (x,y) \mvert \textup{$x$ and $y$ are integers,
 $0 \leq x$ and $0 \leq y < k$}\}$,
and $L$ is isomorphic to $\mbbY_k$, the $k$-row Young's lattice, the lattice of
partitions with at most $k$ parts, with weighting $r/k$ times
the canonical weighting $w(T_{xy}) = k + x - y$.
\end{enumerate}
\end{theorem}
\begin{proof} \disconnect
Note that conclusion case $\mbbY_k$ comprises case~\ref{case:path},
the upward semi-infinite path, as $\mbbY_1$ and the case
lem.~\ref{lem:3-real-cases}(1) as $\mbbY_k$ with $k \geq 2$.
The theorem is the summation of case~\ref{case:path},
lem.~\ref{lem:9-cases}, \ref{lem:xy-ge-3}, \ref{lem:xy-32}, \ref{lem:xy-42},
\ref{lem:xy-62}, and~\ref{lem:3-real-cases}.
\end{proof}

\section{Future directions} \label{sec:future}

The strength of the assumptions of our classification theorem
(th.~\ref{th:classification} is disappointing, but they are
satisfied by all known Fomin lattices except the one family that is not even
distributive, the Young--Fibonacci lattices and
cartesian products that include them.
In future work, we hope to expand the classification to larger classes.

There are a number of ways of relaxing the assumption of
positivity.

\begin{enumerate}
\item
Considering sets of values $V$ that are not the integers $\mbbZ$ may be of
abstract interest but doesn't seem to be directly relevant to
Robinson--Schensted algorithms.
However, investigations of ``generic'' weightings of a lattice can
exploit this
generality.  For instance, looking at weightings of the quadrant as
$P$, we could start with the generic formula
$w((i,j)) = \alpha + \beta r + \gamma i + \delta j$ and see what
values of $\alpha$, $\beta$, $\gamma$, and $\delta$ allow the
differential criterion to be satisfied.
This can be modeled by setting $V = \mbbR^4$ and translating
$\alpha + \beta r + \gamma i + \delta j$ into
$(\alpha,\beta,\gamma,\delta)$ and $r$ into $(0,1,0,0)$.
Once the possible weightings for this $V$ are determined, then
homomorphisms from $V$ to $\mbbZ$ can be considered and the resulting
weightings that are positive can be identified.
From these positive weightings Robinson--Schensted algorithms can be
constructed.

When $L$, $w$, and $V$ are jointly constructed from some complex
algebraic structure (such as \cite{Fom1994a}*{sec.~2.3})
this sort of flexibility may be particularly useful.

\item
Allowing values of $w$ to be negative (when $V = \mbbZ$) does not seem
to be relevant to constructing Robinson--Schensted algorithms.

\item
Allowing values of $w$ to be 0 may be useful as an intermediate step
in constructing positive weightings,
but it will reveal no new Fomin lattices since
th.~\ref{th:truncate-fomin} et al.\ show that a Fomin lattice with
some edges
weighted 0 can be truncated into connected components which are each
Fomin lattices with all edges having nonzero weights.

\item
Allowing $r \leq 0$ (with $V = \mbbZ$) seems to have limited potential as
then either $L$ is trivial or does not have $\zerohat$.
However, see the discussion of Fomin lattices without $\zerohat$ in
(\ref{i:futures:no-zerohat}).

\end{enumerate}

There are also a number of ways of relaxing our structural assumptions
about $L$.

\begin{enumerate}[resume]
\item \label{i:futures:no-zerohat}
Relaxing the assumption that $L$ has a $\zerohat$ seems like it would
be tractable.  But lem.~\ref{lem:graded} and rem.~\ref{rem:graded}
show that the existence of $\zerohat$ has non-obvious consequences in
that it causes unique-cover-modularity to imply that $L$ is graded.
Thus removing the requirement of $\zerohat$ may introduce unexpected
complexity.

Currently, there is no known way to use a Fomin lattice without $\zerohat$
to build a Fomin lattice with $\zerohat$, and thus a
Robinson--Schensted algorithm.  But the truncation theorems
(th.~\ref{th:truncate-fomin} et al.) suggest it's worth looking for
processes by which they might be exploited.

\item
There are a considerable number of Fomin lattices which have all
values of $w$ positive but $r=0$ (and necessarily, no $\zerohat$):
See para.~\ref{para:known}(\ref{i:previous:Z})(\ref{i:previous:CP})%
(\ref{i:previous:SS})(\ref{i:previous:twos}).
There is no obvious application of these to Robinson--Schensted
algorithms, as any cartesian product involving any of them has no
$\zerohat$ and cannot be used in Robinson--Schensted algorithms.  But
this class of lattices may have a different combinatorial significance.

\item
Relaxing the assumption that poset of points is
unique-cover-modular to just cover-modular seems like it may be tractable.
The structure of
neighborhoods would become more complicated but perhaps a careful
analysis of the graphs of possible coverings would show that a
tractable set of constraints characterize the weights of the points.

\item
Relaxing the assumption that the poset
of points is cover-modular seems to be more difficult as it is used in
the proof that no element of $L$ has more than two upward or downward
covers (lem.~\ref{lem:no-triples}) and a number of the lemmas for
global characterization.

\item
Classifying modular but non-distributive Fomin lattices seems to be
the most difficult problem.  There seems to be no
tractable representation of general modular lattices that
parallels Birkhoff's Representation Theorem for distributive lattices,
and our entire analysis is based on characterizing $L$ in the manner
of Birkhoff's representation.

On the other hand, the technique used in
\cite{Stan2012a}*{Exer.~3.51(a) soln.} can be restated to provide a
way to analyze an arbitrary differential distributive lattice, or
conversely, to provide a nondeterministic algorithm for constructing
all differential distributive lattices for a wide range of definitions of
``differential''.  This can be used to prove that no point is covered
by three distinct points (lem.~\ref{lem:no-triples}) and that there is a
unique minimum point (lem.~\ref{lem:P-minimum}).
It may be possible to use this technique to prove additional lemmas
used in this analysis without assuming unique-cover-modularity.
\end{enumerate}

On the largest scale,

\begin{enumerate}[resume]
\item \label{i:why-fomin-lattices}
All of the Robinson--Schensted algorithms that have attracted
attention in the past has been derivable as growth diagrams based on
what we call Fomin lattices, viz., modular lattices whose weightings
are constant on projectivity classes.  However the growth diagram
construction can be based on any pair of dual graded graphs which is a much
broader and more complicated class of structures.  There needs to be
an investigation why only the Fomin lattices have appeared in previous
works.

\item \label{i:sequentially-differential}
Richard Stanley asks how much of this
analysis can be applied to sequentially differential
posets\cite{Stan1990a}*{sec.~2}%
\footnote{Sequentially differential posets have a different differential degree
for each grade of the poset.}
that are distributive lattices.

\item \label{i:cover-function}
Jonathan Farley\cite{Far2003a}*{sec.~11} classifies all finitary
distributive lattices which have a \emph{cover function}, where there
is a function $f: \mbbN \rightarrow \mbbN$ and each
element of the lattice that covers $n$ elements is covered by $f(n)$
elements.  This suggests extending his analysis can be extended to the
weighted case, where
\begin{equation}
  f \left( \sum_{a\mvert a \textup{ maximal in } x} w(a) \right) =
  \sum_{b\mvert b \textup{ minimal in } P \setminus x} w(b)
  \textup{\qquad\qquad for all $x \in \Idealf(P)$}.
\label{eq:cover-distributive}
\end{equation}

\end{enumerate}

\vspace{3em}

\end{document}